\documentclass[12pt]{amsart}

\usepackage[dvips, dvipdfmx]{graphicx}
\usepackage[abs]{overpic}
\usepackage{amssymb}
\usepackage{amsmath}
\usepackage{amsthm}
\usepackage{mathrsfs}
\usepackage{fancybox}
\usepackage{comment}
\usepackage{layout}
\usepackage{color}

\theoremstyle{plain}
\newtheorem{theorem}{Theorem}[section]
\newtheorem{lemma}[theorem]{Lemma}
\newtheorem{corollary}[theorem]{Corollary}
\newtheorem{proposition}[theorem]{Proposition}

\theoremstyle{definition}
\newtheorem{definition}[theorem]{Definition}
\newtheorem{example}[theorem]{Example}
\newtheorem{remark}[theorem]{Remark}

\graphicspath{{figures/}}

\title[HL-homotopy of handlebody-links and Milnor's invariants]{HL-homotopy of handlebody-links and Milnor's invariants}
\author[Kotorii]{Yuka Kotorii}
\address[Kotorii]{Graduate School of Mathematical Science, The University of Tokyo, 3-8-1 Komaba, Meguro-ku, Tokyo 153-8914, Japan}
\author[Mizusawa]{Atsuhiko Mizusawa}
\address[Mizusawa]{Department of Mathematics, Fundamental Science and Engineering, Waseda University, 3-4-1Okubo, Shinjuku-ku, Tokyo 169-8555, Japan} 
\email[Yuka Kotorii]{kotorii@ms.u-tokyo.ac.jp}
\email[Atsuhiko Mizusawa]{a\_mizusawa@aoni.waseda.jp}
\keywords{Handlebody-link, Milnor's $\overline{\mu}$-invariant, clasper, hypermatrix, tensor}
\date{\today}

\begin{document}

\begin{abstract}
A handlebody-link is a disjoint union of embeddings of handlebodies in $S^3$ and an HL-homotopy is an equivalence relation on handlebody-links generated by self-crossing changes.
The second author and Ryo Nikkuni classified the set of HL-homotopy classes of 2-component handlebody-links completely using the linking numbers for handlebody-links.
In this paper, we construct a family of invariants for HL-homotopy classes of general handlebody-links, by using Milnor's $\overline{\mu}$-invariants. Moreover, we give a bijection between the set of HL-homotopy classes of almost trivial handlebody-links and tensor product space modulo some general linear actions, especially for 3- or more component handlebody-links. 
Through this bijection we construct comparable invariants of HL-homotopy classes.
\end{abstract}

\subjclass[2010]{57M27, 57M25}

\maketitle

\section{introduction} \label{sec1}
\par 
A \textit{link} is an embedding of some circles into the 3-sphere $S^3$. A \textit{spatial graph} is a topological embedding of a graph in $S^3$. If all components of a spatial graph are homeomorphic to circles, the spatial graph is regarded as a link. Two spatial graphs are \textit{equivalent} if there is an ambient isotopy which transform one to the other. 
\par
A \textit{handlebody-link} \cite{Ishi, Su} is a disjoint union of embeddings of handlebodies in  $S^3$ (Figure \ref{fig1}). 
\begin{figure}[ht] 
$$
\raisebox{7 pt}{\includegraphics[bb= 0 0 169 93, height=35 pt]{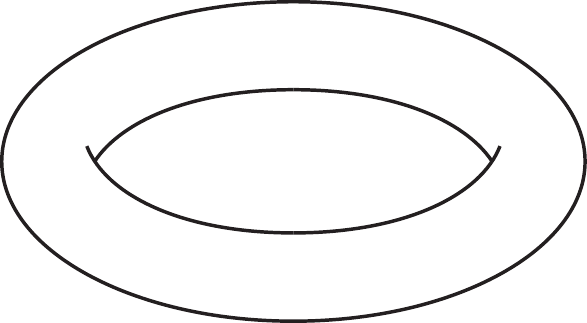}} 
\hspace{0.4cm}
\raisebox{7 pt}{\includegraphics[bb= 0 0 328 93, height=35 pt]{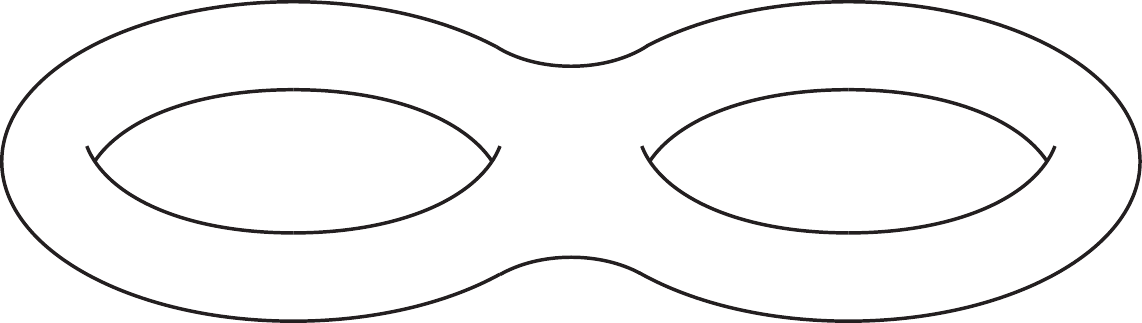}} 
\hspace{-6.7cm}\raisebox{-35 pt}{\includegraphics[bb= 0 0 488 93, height=35 pt]{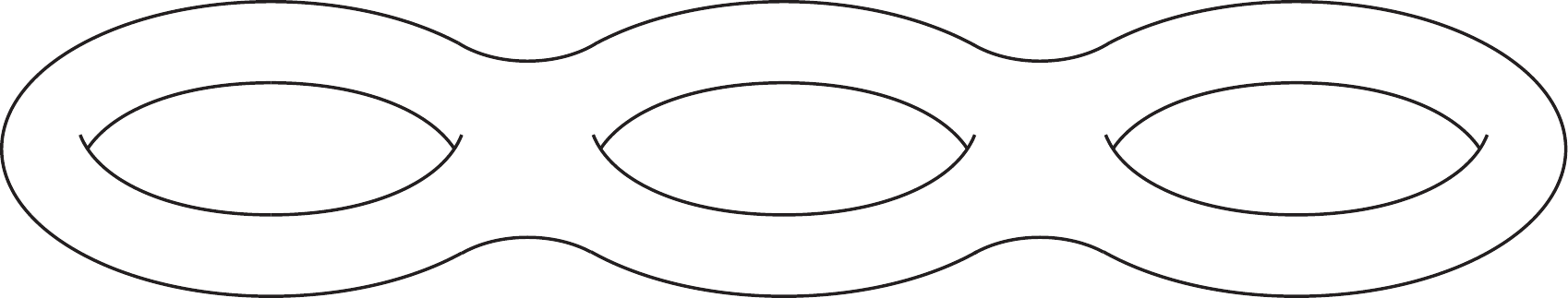}} 
\hspace{0.4cm}\underset{\mbox{\footnotesize }}{ \mbox{ \LARGE{$\hookrightarrow$}}} \hspace{0.2cm}
\raisebox{47 pt}{\begin{overpic}[bb=0 0 54 54, height=20 pt]{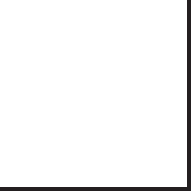}\put(1,6){$S^3$}\end{overpic}} \hspace{0.2cm} 
\raisebox{-47 pt}{\includegraphics[bb= 0 0 225 154, height=110 pt]{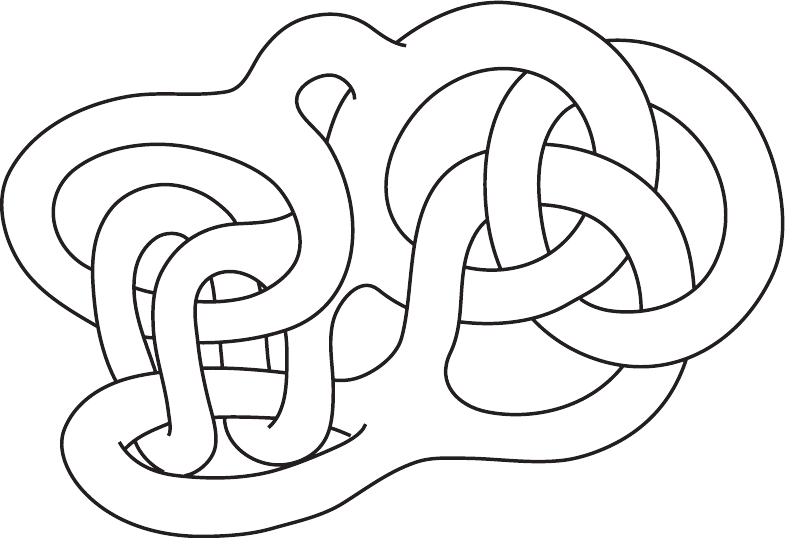}}
$$
\caption{Handlebody-link.} \label{fig1}
\end{figure}
Two handlebody-links are \textit{equivalent} if there is an ambient isotopy which transforms one to the other. A handlebody-link can be represented by a spatial graph. A spatial graph $G$ is said to \textit{represent} a handlebody-link $H$ if the regular neighborhood of $G$ is ambient isotopic to $H$. There are infinitely many spatial graphs which represent the same handlebody-link. It is known that two spatial graphs which represent the same handlebody-link are transformed to each other by a sequence of contraction moves in Figure \ref{cons}, which is to contract an edge connecting two different vertices and its inverse (See \cite{Ishi} for details). A handlebody-link is \textit{trivial} if it is represented by a plane graph.
\begin{figure}[ht]
$$
\raisebox{-26 pt}{\begin{overpic}[bb=0 0 499 316, height=63pt]{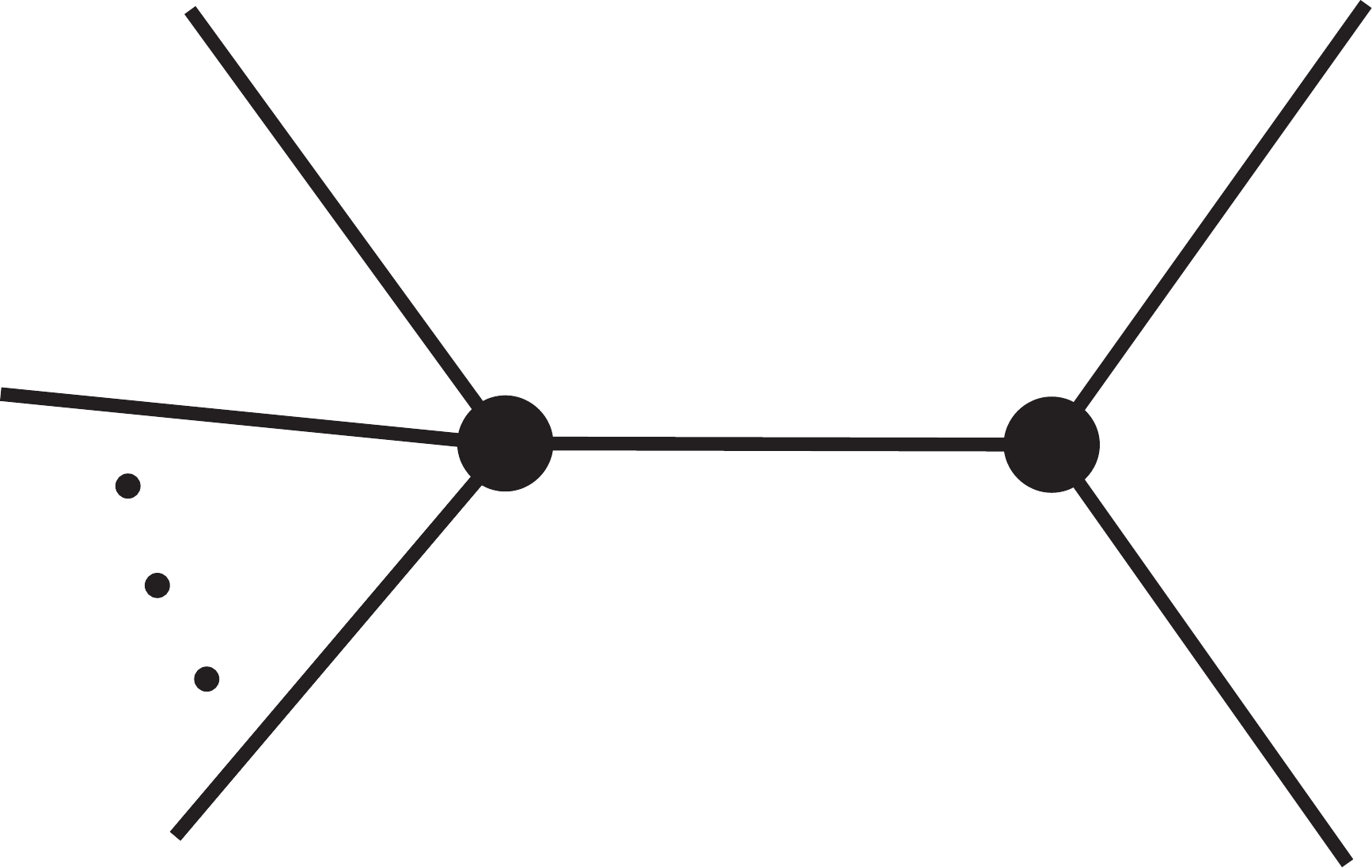}
\end{overpic}}
\hspace{0.4 cm}\mbox{\LARGE{$\leftrightarrow$}}\hspace{0.5 cm}
\raisebox{-26 pt}{\includegraphics[bb=0 0 300 316, height=63 pt]{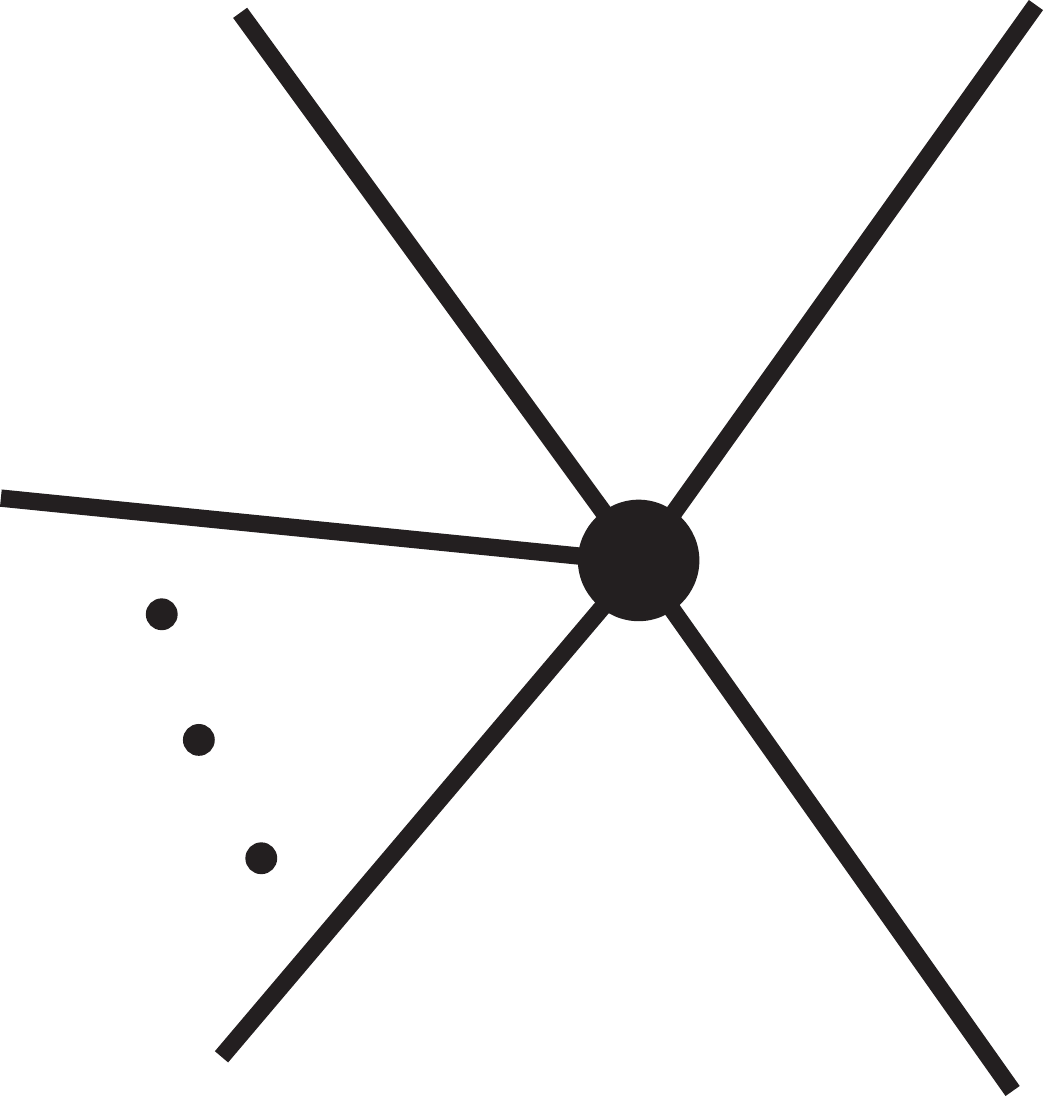}}$$
\caption{Contraction move.} \label{cons}
\end{figure}
\par
A \textit{self-crossing change} of a link (resp. a spatial graph) is a crossing change of two arcs which belong to the same component of a link (resp. a spatial graph). J.~Milnor defined a class of links called a \textit{link-homotopy} \cite{Mil01}. Two links are \textit{link-homotopic} if they are transformed to each other by self-crossing changes and ambient isotopies. The notion of link-homotopy was generalized for spatial graphs. An \textit{edge-homotopy} (resp. a \textit{vertex homotopy}, a \textit{component homotopy}) is an equivalence relation generated by ambient isotopies and crossing changes of two arcs which  belong to the same edge (resp. adjacent edges, the same component). We generalize the notion of link-homotopy to handleboody-links. \begin{definition}[HL-homotopy] 
Let $H_0$ be $n$ handlebodies and $H_i$ $(i=1, 2)$ two $n$-component handlebody-links obtained by embeddings $f_i$'s of $H_0$ to $S^3$. Two handlebody-links $H_1$ and $H_2$ are called \textit{HL-homotopic} if there is homotopy $h_t$ from $f_1$ to $f_2$ where the components of $h_t(H_0)$ are mutually disjoint at any $0\leq t \leq 1$. 
\end{definition}
\begin{remark}
In \cite{MN}, the notation of \textit{neighborhood homotopy} of spatial graphs was introduced. Two spatial graphs $G$ and $G'$ are neighborhood homotopic if they are transformed to each other by a sequence of ambient isotopies, contraction moves and self-crossing changes. 
Let $H_1$ and $H_2$ be two handlebody-links and let $G_1$ and $G_2$ be spatial graph presentations of $H_1$ and $H_2$ respectively. The handlebody-links $H_1$ and $H_2$ are HL-homotopic if and only if $G_1$ and $G_2$ are neighborhood homotopic. 
\end{remark}
\par
J.~Milnor classified the link-homotopy classes of 2-component links by linking numbers \cite{Mil01}. Moreover, 
he defined a family of invariants for an ordered oriented link in $S^3$ as a generalization of the linking numbers, in \cite{Mil01, Mil02}. These invariants are called {\em Milnor's $\overline{\mu}$-invariants}. For an ordered oriented $n$-component link $L$, Milnor's $\overline{\mu}$-invariant is specified by a sequence $I$ of indices in $\{1,2,\ldots ,n\}$ and denoted by $\overline{\mu}_L(I)$.
If the sequence is with distinct indices, then this invariant is also link-homotopy invariant and called {\em Milnor's link-homotopy invariant}. He also classified the link-homotopy classes of 3-component links by the link-homotopy invariants $\overline{\mu}_L(I)$ with $|I|=3$, where $|I|$ is the length of $I$. The link-homotopy classes of 4-component links were classified by \cite{L}. In general, there is an algorithm which determines whether 
two link are link-homotopic or not \cite{HL}.
\par 
For handlebody-links, by generalizing invariants of links or spatial graphs, some invariants are defined; for example, quandle invariants \cite{IshiIwa}. The second author defined \textit{linking numbers} for 2-component handlebody-links by generalizing the linking number of links.
\begin{definition}[Linking numbers \cite{Miz}] 
Let $L_1 \cup L_2$ be a 2-component handlebody-link and $\{e^i_1, \dots, e^i_{g_i}\}$ a basis of the first homology group $H_1(L_i;\mathbb{Z})$ of $L_i$ where $g_i$ is the genus of $L_i$. An element $e^i_j$ is considered as an oriented closed circle in $S^3$. Let $M$ be a matrix whose $(j,k)$-entry is the linking number $lk(e^1_j, e^2_k)$ and $d_1|d_2|\dots|d_l$ the elementary divisors of $M$. Then \textit{linking numbers} $Lk(L_1,L_2)$ of $L_1\cup L_2$ is defined as the multiset of absolute values of the elementary divisors $\{|d_1|
, |d_2|, \dots, |d_l|\}$.
\end{definition}
\par
In \cite{MN}, the HL-homotopy classes of 2-component handlebody-links are classified completely by the linking numbers for handlebody-links. 
\begin{theorem}[\cite{MN}] 
 Let $L_1\cup L_2$ and $L'_1\cup L'_2$ be 2-component handlebody-links such that the genera of $L_i$ and $L'_i$ are equal for each $i$. Then $L_1\cup L_2$ and $L'_1\cup L'_2$ are HL-homotopic if and only if $Lk(L_1, L_2)=Lk(L'_1, L'_2)$.
\end{theorem}
\par
In this paper, handlebody-links whose components are more than two are considered. Using Milnor's link-homotopy invariants, we make a map from the set of handlebody-links to a tensor product space of $\mathbb{Z}$-modules up to some actions. By using this, we give a necessary and sufficient condition of that a handlebody-link is HL-homotopic to the trivial one. When handlebody-links are almost trivial, the map give a bijection between HL-homotopy classes of handlebody-links and some tensor product space up to some actions. Through the map we construct comparable invariants of HL-homotopy classes. 
\par
This paper is organized as follows. In section 2, we introduce a bouquet graph presentation of handlebody-links. Section 3 and 4 review Milnor's $\overline{\mu}$-invariants and hypermatrices respectively. 
In section 5, we state the main theorem. We construct a family of HL-homotopy invariants for general handlebody-links by using Milnor's link-homotopy invariant. We then construct a bijection between the HL-homotopy classes of almost trivial handlebody-links and tensor product spaces up to some modulo. Moreover, we define some comparable invariants for HL-homotopy classes of almost trivial handlebody-links. 
Section 6 shows the proof of the main theorem by using the clusper theory.
In Appendix A, we review a proof of Theorem \ref{EST} in Section 2.

\section*{Acknowledgements} 
The authors thank Professor Kouki Taniyama for valuable comments and suggestions.
They also thanks Professor Ryo Nikkuni for useful discussions and comments. 
They also thanks Professor Akira Yasuhara for useful comments. 
They also thanks Professor Jun Murakami for valuable advices.
They also thanks Professor Atsushi Ishii for some advices.

\section{Bouquet graph presentation of handlebody-links} \label{sec2}
\par
 A bouquet graph is a graph which has only one vertex. We can represent each component of a handlebody-link by an embedding of a bouquet graph. We call the presentation a \textit{bouquet graph presentation}. A bouquet graph presentation of a handlebody-link is obtained from a spatial graph presentation $G$ of the handlebody-link by contracting a spanning tree of each component of $G$ to a point. The following fact is known.
\begin{theorem}[{\cite{MS}}] \label{EST}
Two bouquet graph presentations which represent the same handlebody-link are transformed to each other by a sequence of edge-slides in Figure \ref{ES}, which slides an end point of an loop edge along another edge from the vertex to the vertex.
\end{theorem}
\begin{proof} See Appendix \ref{app1}.
\end{proof}
\begin{figure}[ht]
$$
\raisebox{-30 pt}{\includegraphics[bb= 0 0 196 267, width=50 pt]{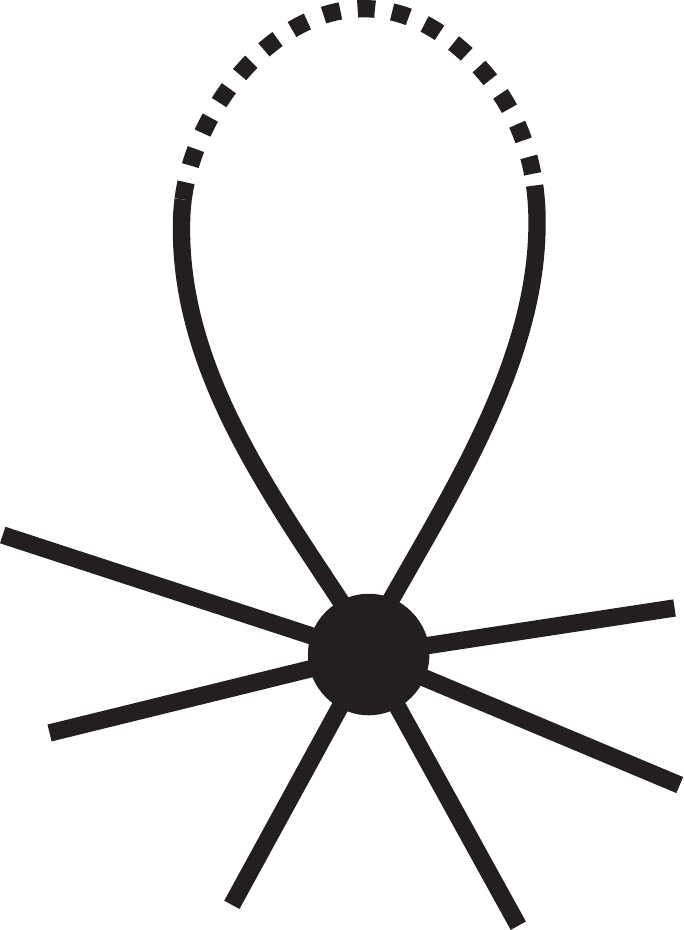}}
\hspace{0.3 cm}\mbox{\LARGE{$\leftrightarrow$}}\hspace{0.3 cm}
\raisebox{-30 pt}{\includegraphics[bb=0 0 196 291, width=50 pt]{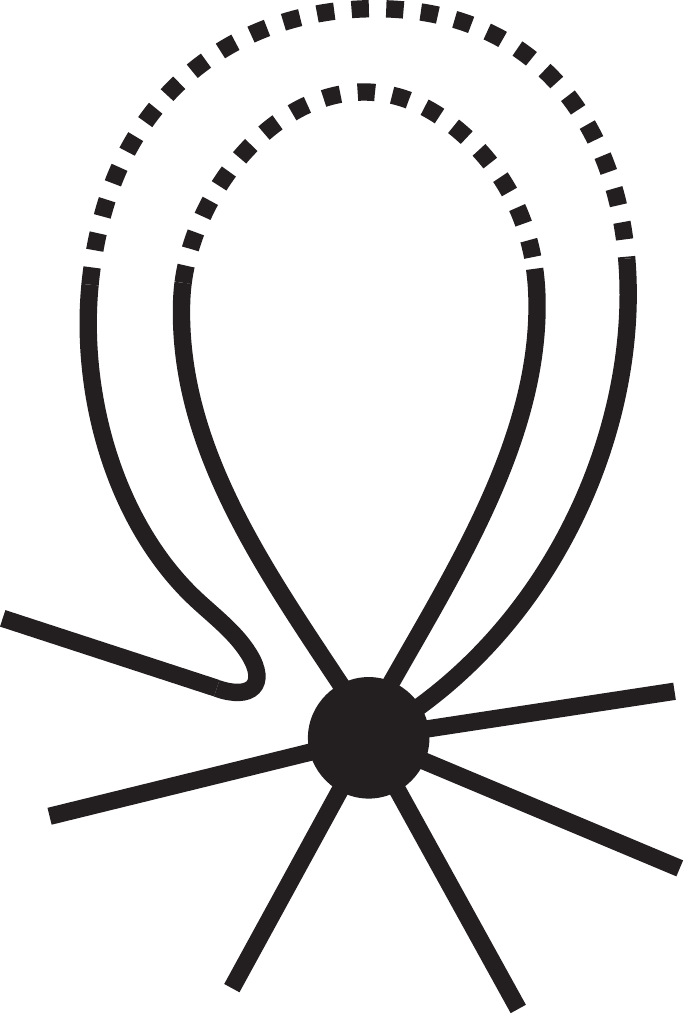}}
$$
\caption{Edge-sliding move.}\label{ES}
\end{figure}
\begin{proposition} \label{HandB}
Let $H$ and $H'$ be handlebody-links and let $\Gamma$ and $\Gamma'$ be bouquet graph presentations of $H$ and $H'$ respectively. Then $H$ and $H'$ are HL-homotopic if and only if $\Gamma$ and $\Gamma'$ are transformed to each other by a sequence of ambient isotopies, edge-slides and self-crossing changes.
\end{proposition}
\begin{proof} The ``if" part is obvious. We prove the ``only if" part.
Since $H$ and $H'$ are HL-homotopic, there is a sequence of spatial
graphs $G_1, \dots, G_n$, where $G_1$ and $G_n$ represent $H$ and $H'$
respectively and $G_{i+1}$ is obtained  from $G_i$ by a contraction
move or a self-crossing change. Let $H_i$ be a handlebody-link which
is represented by $G_i$. We have a sequence $\Gamma_1, \dots,
\Gamma_n$ of bouquet graph presentations of $H_1, \dots, H_n$ from
$G_1, \dots, G_n$ as follows. Fix a spanning tree $T_i$ of $G_i$ and
fix a vertex $v_i$ of $T_i$. Let $N(T_i)$ be the regular neighborhood
of $T_i$ in $H_i$. Contracting $T_i$ in $N(T_i)$ to $v_i$, we have a
bouquet graph presentation $\Gamma_i$ of $G_i$.
\par
Suppose that, for some $i$, $G_{i+1}$ is obtained from $G_i$ by a
self-crossing change of two edges $e_1$ and $e_2$. Let $T'_{i+1}$
be a spanning tree of $G_{i+1}$ obtained from $T_i$ by the
self-crossing change. By contracting $T'_{i+1}$ as above, we have a
bouquet graph presentation $\Gamma'_{i+1}$ of $H_{i+1}$. Then
$\Gamma'_{i+1}$ is obtained from $\Gamma_{i}$ by some self-crossing
changes corresponding to the self-crossing change of $e_1$ and
$e_2$.
 From Theorem \ref{EST}, since both $\Gamma_{i+1}$ and $\Gamma'_{i+1}$
represent $H_{i+1}$, $\Gamma_{i+1}$ is obtained from $\Gamma'_{i+1}$
by using edge-slides.
\par
If $G_{i+1}$ is obtained from $G_i$ by a contraction move, $\Gamma_i$
and $\Gamma_{i+1}$ represents the same handlebody-link therefore
$\Gamma_{i+1}$ is obtained from $\Gamma_i$ by using edge-slides. Then
$\Gamma_n$ is obtained from $\Gamma_1$ by using edge-slides and
self-crossing changes. Similarly, $\Gamma_1$ (resp. $\Gamma'$) is
obtained from $\Gamma$ (resp. $\Gamma_n$) by edge-slides. This
completes the proof.
\end{proof}
We note that for a bouquet graph presentation $\Gamma$ of a handlebody-link $H$, by giving arbitrary orientation to each loop edge of $\Gamma$, the loop edges of $\Gamma$ represent a basis of the first homology group $H_1(H; \mathbb{Z})$ of $H$. 
\section{Milnor's Invariant} \label{sec3}
We introduce the definition of Milnor's link-homotopy invariants, and 
to give invariants for handlebody-link in Section \ref{sec5},  
we show that these are additive under a bund sum for components. 

\subsection{Definition of Milnor's link-homotopy invariant}

Let $L=L_1 \cup \dots \cup L_n$ be an ordered oriented $n$-component link in ${S}^3$.
Consider the link group $\pi = \pi_1({S}^3 \setminus L_1 \cup \dots \cup L_{n-1})$ of $L_1 \cup \dots \cup L_{n-1}$ and denote the $i$-th meridian by $m_i$ for $i$ ($1\leq i \leq n-1$).

Given a finitely generated group $G$, the {\it reduced group} $\overline{G}$ is defined to the quotient of $G$ by its normal subgroup generated by $[g, hgh^{-1}]$ for any $g, h \in G$,
where $[a,b]$ means the commutator of $a$ and $b$.
Then $\overline{\pi}$ is generated by the meridians $m_1, m_2, \dots ,m_{n-1}$.

Let $\mathbb{Z}[[X_1, \dots, X_{n-1}]]$ be the non-commutative formal power series ring generated by $X_1, \dots, X_{n-1}$.
Denote by $\hat{Z}$ its quotient ring by the two-side ideal generated by all monomials in which at least one of the generators appear at least twice.
The {\it Magnus expansion} $\varphi$ is a homomorphism from the free group $F(m_1, \dots , m_{n-1})$ generated by $m_1, \dots , m_{n-1} $ into $\mathbb{Z}[[X_1, \dots, X_{n-1}]]$, defined by sending $m_i$ to $1+ X_i$ and $m_i^{-1}$ to $1-X_i+X_i^2 - \cdots $.
It induces a homomorphism from $\overline{F(m_1, \dots , m_{n-1})}$ into $\hat{Z}$.
Let $w_n \in {F(m_1, \dots , m_{n-1})}$ be a word representing $L_n$ in $\overline{\pi}$.
We then define ${\mu}_L(i_1 i_2 \ldots i_r n) $ for distinct indices $i_1, i_2, \dots , i_r, n $ as the coefficient of the Magnus expansion of $w_n$ in $\hat{Z}$:
\[ \varphi (w_n) = 1 + \sum {\mu}_L(i_1 i_2 \ldots i_r n) X_{i_1} X_{i_2} \dots X_{i_r}, \]
where the summation is over all sequences $i_1i_2 \ldots i_r$ with distinct indices between 1 and $n-1$. 
Similarly, we define ${\mu}_L(i_1 i_2 \ldots i_s)$ for any distinct indices between 1 and $n$. 
We define $\overline{\mu}_L(i_1 i_2 \ldots i_r n) $ as the residue class of ${\mu}_L(i_1 i_2 \ldots i_r n) $ modulo the indeterminacy $\Delta_L(i_1 i_2 \ldots i_r n) $ which is the greatest common divisor of   ${\mu}_L(j_1 j_2 \ldots j_s)$'s, where $j_1 j_2 \ldots j_s$ ranges over all sequences obtained by deleting at least one of the indices $i_1,i_2, \dots ,i_r, n$ and permuting the remaining ones cyclicly.
Moreover we define $\Delta_L(i_1 n)=0$.

\begin{theorem}[\cite{Mil01,Mil02}]
If $L$ and $L'$ are link-homotopic, then $\overline{\mu}_L(I) = \overline{\mu}_{L'}(I)$ for any sequence $I$ with distinct indices.
\end{theorem}

\begin{lemma}[\cite{Mil02}] \label{Milnor'sLemma}
Let $L$ be an ordered oriented link.  Then the following relations hold. \\
(1) $\overline{\mu}_{L}{(i_1i_2 \ldots i_m)} = \overline{\mu}_{L}{(i_2 \ldots i_mi_1)}$  \\
(2) If the orientation of the $k$-th component of $L$ is reversed, then $\overline{\mu}_{L}{(i_1i_2 \ldots i_m)}$ is multiplied by $-1$ or $+1$ according as the sequence $i_1i_2 \ldots i_m$ contains $k$ once or not.
\end{lemma}

\subsection{Additivity property of Milnor's link-homotopy invariant}

The following lemma is used for Section \ref{sec5}, which is showed by using the definition of Milnor's link-homotopy invariants.

\begin{lemma}\label{add}
Let $L=L_1 \cup L_2 \cup \dots \cup L_{n-1}$ be an $(n-1)$-component link in $S^3$.
Let $K$ and $K'$ be disjoint knots in $S^3 \setminus L$.
Let $I$ be a sequence with distinct indices in $\{ 1, 2, \dots, n \}$.
If $I$ contains the index $n$,
$$\mu_{L \cup (K \sharp_b K')} (I) \equiv  \mu_{L \cup K} (I) + \mu_{L \cup K'} (I)  
\mod{ \gcd(\Delta_{L \cup K}(I), \Delta_{L \cup K'}(I))}, 
$$
where $K \sharp_b K'$ is a band sum of $K$ and $K'$ with respect to any band,
and $L \cup (K \sharp_b K')$, $L \cup K$ and $L \cup K'$ are $n$-component links whose $n$-th components are $K \sharp_b K'$, $K$ and $K'$, respectively.
\end{lemma}

\begin{remark}
By a property of the $\bar{\mu}$-invariant, we can obtain the same result for a band sum of the $i$-th component instead of the $n$-th component. 
\end{remark}

\begin{proof}
Assume the last index of $I$ is $n$. 
Choosing words $w_n$, $w'_n$ in $F(m_1, \dots ,m_{n-1})$ representing $K$, $K'$ in $\overline{\pi_1({S}^3 \setminus L)}$ appropriately, we may assume $w_n \cdot w'_n$ is a word representing $K \sharp_b K'$.
By the definition of indeterminacy ${\Delta}$, we have that 
 \begin{align*}
\mu_{L \cup (K \sharp_b K')} {(I)} 
& \equiv \sum \mu_{L \cup K} {(I_1 n)} \cdot \mu_{L \cup K'} {(I_2 n)}  \\
& \equiv \mu_{L \cup K}{(I)}   + \mu_{L \cup K'}{(I)} \\
& \mod{ \gcd(\Delta_{L \cup K}(I), \Delta_{L \cup K'}(I), \Delta_{L \cup (K \sharp_b K')}(I))}, 
\end{align*}
where the summation is over all sequences $I_1$, $I_2$ such that the sequence $I_1 I_2 n$ is equal to $I$ (it is possible that either $I_1$ or $I_2$ is empty).
By Lemma~\ref{Milnor'sLemma} (1), we have that for any sequence $I$ which contains the index $n$,
\begin{align*}
\mu_{L \cup (K \sharp_b K')} {(I)}  
 \equiv & \mu_{L \cup K}{(I)}  + \mu_{L \cup K'}{(I)}   
\mod{ \gcd(\Delta_{L \cup K}(I), \Delta_{L \cup K'}(I), \Delta_{L \cup (K \sharp_b K')}(I))}.
\end{align*}

It remains to prove that for any sequence $I$ with distinct indices in $\{ 1,2, \dots ,n \}$, 
\begin{align}\label{induction}
\gcd(\Delta_{L \cup K}(I), \Delta_{L \cup K'}(I), \Delta_{L \cup (K \sharp_b K')}(I)) 
= \gcd(\Delta_{L \cup K}(I), \Delta_{L \cup K'}(I)). 
\end{align}
The proof is by induction on the length $m$ of sequence $I$.
For the case $m=2$, $\Delta_{L \cup K}(I)$, $\Delta_{L \cup K'}(I)$ and $\Delta_{L \cup (K \sharp_b K')}(I)$ are all zero, by the definition of $\Delta$.
Assume the formula (\ref{induction}) holds for $m-1$, we will prove it for $m$.
Let the length of $I$ be $m$ and $J$ be a subsequence of $I$ whose length is $m-1$. 
If $J$ does not contain the index $n$, 
$$\mu_{L \cup (K \sharp_b K')} {(J)} \equiv \mu_{L \cup K}{(J)} \mod \Delta_{L \cup K}(J) (= \Delta_{L \cup (K \sharp_b K')}(J)). $$
If $J$ contains the index $n$,  by the assumption of induction, 
$$\mu_{L \cup (K \sharp_b K')} (J) \equiv  \mu_{L \cup K} (J) + \mu_{L \cup K'} (J)  
\mod{ \gcd(\Delta_{L \cup K}(J), \Delta_{L \cup K'}(J))}.$$
Therefore  $\mu_{L \cup (K \sharp_b K')} (J)$ can be divided by $\gcd(\Delta_{L \cup K}(I), \Delta_{L \cup K'}(I))$ for any $J$.
Thus $\Delta_{L \cup (K \sharp_b K')}(I)$ can be divided by $\gcd(\Delta_{L \cup K}(I), \Delta_{L \cup K'}(I))$, and the formula (\ref{induction}) holds for $m$.
\end{proof}  

\begin{remark}
In \cite{K}, V. S.~Krushkal showed Milnor's $\overline{\mu}$-invariants are additive under connected sum for links which are separated by a 2-sphere.    
\end{remark}

\section{HyperMatrix} \label{sec4}
We introduce a hypermatrix and define its three transformations. We can identify a hypermatrix with a tensor.

\subsection{Hypermatrix and transformations} 
A {\textit hypermatrix} is a generalization of a matrix to higher dimensions, i.e. numbers arranged in hyperrectangle form. A hypermatrix is defined rigorously as follows.
For $m_1, \dots , m_d \in \mathbb{N}$, a map $f :\langle m_1 \rangle \times \dots \times \langle m_d \rangle \to \mathbb{Z}_{\delta}$ is a {\it d-hypermatrix} of size $(m_1, \dots , m_d)$ with coefficients in $\mathbb{Z}_{\delta}$,
where $\langle m_i \rangle = \{ 1,2, \dots ,m_i\}$, $\delta$ is a non-negative integer and $\mathbb{Z}_{\delta}=\mathbb{Z}$ if $\delta=0$ otherwise $\mathbb{Z}_{\delta}=\mathbb{Z}/\delta\mathbb{Z}$. 
Denote $ f ({k_1}, \dots  ,{k_d})$ by $a_{{k_1} \ldots {k_d}}$, and $f$ by $A=(a_{{k_1} \ldots {k_d}})_{k_1, k_2, \dots ,k_d=1}^{m_1, m_2, \dots ,m_d}$.
The set of $d$-hypermatrices of size $(m_1, \dots , m_d)$ with $\mathbb{Z}_{\delta}$-coefficients is denoted by $M(m_1, \dots , m_d; \mathbb{Z}_{\delta})$ and its element is called an $m_1 \times \dots \times m_d$ hypermatirx. If $\delta=0$, we simplify $M(m_1, \dots , m_d; \mathbb{Z})$ to $M(m_1, \dots , m_d)$.
A $3$-hypermatrix $A=(a_{jkl})_{j,k,l=1}^{4,3,2} \in M(4,3,2)$ can be written down as two slices of $4 \times 3$ matrices
\[
  A = \left(
    \begin{array}{ccc}
      a_{111} & a_{121}  & a_{131} \\
      a_{211} & a_{221} &  a_{231} \\
      a_{311} & a_{321} &  a_{331} \\
      a_{411} & a_{421} &  a_{431} \\
    \end{array}
  \right| 
  \left.
\begin{array}{ccc}
      a_{112} & a_{122}  & a_{132} \\
      a_{212} & a_{222} &  a_{232} \\
      a_{312} & a_{322} &  a_{332} \\
      a_{412} & a_{422} &  a_{432} \\    
      \end{array}
  \right) \in M(4,3,2),
\]
where $j$, $k$ and $l$ index the row, column and slice, respectively.  

Applying ${\rm GL}(m_i, \mathbb{Z})$ action to the $i$-th coordinate, we have the following three transformations of $M(m_1,m_2, \dots ,m_d; \mathbb{Z}_{\delta})$.
\begin{itemize}
\item[(TI)] For fixed $j_1, j_2 \in \{ 1, 2, \dots, m_i\}$ $(j_1 \neq j_2)$, exchanging entries $a_{{k_1} \ldots k_{i-1} j_1 k_{i+1} \ldots {k_d}}$ and $a_{{k_1} \ldots k_{i-1} j_2 k_{i+1} \ldots {k_d}}$ 
for every $k_1, \dots , k_{i-1}, k_{i+1}, \dots , {k_d}$. 
\item[(TII)] For a fixed $j \in \{ 1, 2, \dots , m_i\}$, multiplying $\pm 1$ to an entry $a_{{k_1} \ldots k_{i-1} j k_{i+1} \ldots {k_d}}$ for every $k_1, \dots , k_{i-1}, k_{i+1}, \dots , {k_d}$. 
\item[(TIII)] For fixed $j_1, j_2 \in \{ 1, 2, \dots , m_i\}$ $(j_1 \neq j_2)$, adding an integer multiplied entry $a_{{k_1} \ldots k_{i-1} j_1 k_{i+1} \ldots {k_d}}$ to another entry $a_{{k_1} \ldots k_{i-1} j_2 k_{i+1} \ldots {k_d}}$ for every $k_1, \dots , k_{i-1}, k_{i+1}, \dots , {k_d}$.
\end{itemize}
If $\delta=0$ (i.e. coefficients are integers), these transformations coincide with the \textit{elementary transformations}.

\begin{remark} \label{THid}
A $\mathbb{Z}_{\delta}$-module $\mathbb{Z}_{\delta}^{m_1} \otimes_{\mathbb{Z}}  \dots \otimes_{\mathbb{Z}}\mathbb{Z}_{\delta}^{m_d}$ can be identified with  $M({m_1, \dots , m_d}; \mathbb{Z}_{\delta})$ as follows. 
For any $T$ in $\mathbb{Z}_{\delta}^{m_1} \otimes_{\mathbb{Z}} \dots \otimes_{\mathbb{Z}} \mathbb{Z}_{\delta}^{m_d}$, it can be represented by
$$T = \sum_{k_1, \dots ,k_d=1}^{m_1,\dots ,m_d} a_{k_1 \ldots k_d} {\overline{\boldsymbol e}}_{k_1}^1 \otimes \dots \otimes \overline{\boldsymbol e}_{k_d}^d,$$
where $a_{k_1 \ldots k_d}\in \mathbb{Z}_{\delta}$, and ${\overline{\boldsymbol e}}_{k_i}^i$ is an element of basis of  $\mathbb{Z}_{\delta}^{m_i}$ as a $\mathbb{Z}_{\delta}$-module for any $i$ ($1 \leq i \leq d$) and $k_i$ ($1 \leq k_i \leq m_i$).
By a map from $T$ to $(a_{{k_1} \ldots {k_d}})_{k_1, \dots ,k_d=1}^{m_1, \dots ,m_d}$,
a tensor can be identified with a hypermatrix.
\end{remark}

\section{Main Theorem} \label{sec5}
We give a map from handlebody-links to a tensor product space by using Milnor's $\overline{\mu}$-invariant. 
This map induces a necessary and sufficient condition of that a handlebody-link is HL-homotopic to the trivial one, and a bijection from HL-homotopy classes of almost trivial handlebody-links to a union of direct sums of the tensor product spaces modulo some action. Through the bijection, we define some comparable invariants for HL-homotopy classes of handlebody-links.

\subsection{$\overline{\mu}$-invariant for handlebody-links} \label{mapdef}
\par
Let $H=L_1 \cup  \dots \cup L_n$ be an $n$-component handlebody-link with genus $g_i$ for each $i$.
Let $\mathcal{B}=\{ e_{1}^{1}, \dots , e_{g_1}^{1}, \dots ,  e_{1}^{n} \dots , e_{g_n}^{n} \}$ be a basis of the first homology group $H_1(H; \mathbb{Z})$ where $e_{j}^{i}$ is the $j$-th element of the basis of $H_1(L_i,\mathbb{Z})$.
We can regard the basis as embedded closed oriented circles in ${S}^3$.
Therefore $e_{k_1}^{1} \cup e_{k_2}^{2} \cup \dots \cup e_{k_n}^{n}$ can be regarded as a link for each $k_i$ ($1 \leq k_i \leq g_i$).
Let $I=i_1i_2 \dots i_m$ ($m \leq n$) be a sequence with distinct indices in $\{ 1, 2, \dots , n \}$.  
For each $I$, we define an element $t_{H, \mathcal{B}}(I) \in {(\mathbb{Z}_{{\Delta}_I})}^{g_{1}} \otimes \dots \otimes {(\mathbb{Z}_{{\Delta}_I})}^{g_{n}}$ as 
$$ t_{H,\mathcal{B}}(I) := \sum_{{k_1}, \dots ,{k_n}=1}^{g_{1}, \dots, g_{n}}  \overline{\mu}_{e_{{k_1}}^{1} \cup  \dots \cup e_{{k_n}}^{n} }{(I)} \ {\overline{\boldsymbol{e}}_{{k_1}}^{1} \otimes \dots \otimes \overline{\boldsymbol{e}}_{{k_n}}^{n}},$$
where $\overline{\mu}_{e_{k_{1}}^{1} \cup \dots \cup e_{k_{n}}^{n} }{(I)}\in \mathbb{Z}_{{\Delta}_I}$, ${\Delta}_I$ is the greatest common divisor of all ${\mu}_{e_{k_{1}}^{1} \cup  \dots \cup e_{k_{n}}^{n} }{(J)}$ for all $k_1, \dots, k_n$ and $J$ obtained from $I$ by deleting at least one index and permuting the remaining ones cyclically, and $\overline{\boldsymbol{e}}^{i}_{k_i}$ is a fixed basis
of $(\mathbb{Z}_{{\Delta}_I})^{g_i}$ as $\mathbb{Z}_{{\Delta}_I}$-module. Since Milnor's $\overline{\mu}$-invariants are link-homotopy invariants, $t_{H,\mathcal{B}}(I)$ is invariant under an HL-homotopy preserving $\mathcal{B}$.

\begin{remark}
(1) For an oriented spatial graph $\Gamma$ whose components are all bouquet graphs, we can define an element $t_{\Gamma}(I) \in {(\mathbb{Z}_{{\Delta}_I})}^{g_{1}} \otimes \dots \otimes {(\mathbb{Z}_{{\Delta}_I})}^{g_{n}}$ as above, where $g_i$ is a number of edges of the $i$-th component of $\Gamma$. Let $e^i_j$ be the $j$-th loop edge of the $i$-th component of $\Gamma$. Then we define 
$$t_{\Gamma}(I) := \sum_{{k_1}, \dots ,{k_n}=1}^{g_{1}, \dots, g_{n}}  \overline{\mu}_{e_{{k_1}}^{1} \cup  \dots \cup e_{{k_n}}^{n} }{(I)} \ {\overline{\boldsymbol{e}}_{{k_1}}^{1} \otimes \dots \otimes \overline{\boldsymbol{e}}_{{k_n}}^{n}}.$$
If $\Gamma$ is an oriented bouquet graph presentation of a handlebody-link $H$ with respect to a basis $\mathcal{B}$, then $t_{\Gamma}(I) =t_{H,\mathcal{B}}(I)$.
\par
(2) An indeterminacy ${\Delta}_I$ coincides with the greatest common divisor of all ${\Delta}_{e_{k_1}^{1} \cup \dots \cup e_{k_n}^{n} }{(I)}$ for all $k_1, \dots, k_n$, where ${\Delta}_{e_{k_1}^{1} \cup \dots \cup e_{k_n}^{n} }{(I)}$ is an indeterminacy of the original Milnor's invariant for the link $e_{k_1}^{1} \cup e_{k_2}^{2} \cup \dots \cup e_{k_n}^{n}$.
We assume that the first homology group of each component of $H$ is $\mathbb{Z}$, 
i.e. all genera of components of $H$ are 1.
Then, $t_{H, \mathcal{B}}(I)$ is identified with the original Milnor's link-homotopy invariant for a link, essentially. 
\end{remark}

We consider a general linear group action to tensors. For any $t$ in a $\mathbb{Z}_{{\Delta}_I}$-module ${(\mathbb{Z}_{{\Delta}_I})}^{g_{1}} \otimes \dots \otimes {(\mathbb{Z}_{{\Delta}_I})}^{g_{n}}$,
we can represent it as
\[t = \sum_{{k_1}, \dots ,{k_n}=1}^{g_{1}, \dots ,g_{n}} a_{{{k_1}}  \ldots  {k_n}} \ {\overline{\boldsymbol{e}}_{{k_1}}^{1} \otimes \dots \otimes \overline{\boldsymbol{e}}_{{k_n}}^{n}}, \] 
where $a_{{{k_1}}  \ldots  {k_n}} \in \mathbb{Z}_{{\Delta}_I}$. 
We define an action $\rho$ of $GL(g_{1}, \mathbb{Z}) \times  \dots \times GL(g_{n}, \mathbb{Z})$ on ${(\mathbb{Z}_{{\Delta}_I})}^{g_{1}} \otimes \dots \otimes {(\mathbb{Z}_{{\Delta}_I})}^{g_{n}}$ as follows.
For any $A_i$ in $GL(g_{i}, \mathbb{Z})$, 
\[ (A_1,  \dots, A_n) \cdot t = \sum_{{k_1}, \dots ,{k_n}=1}^{g_{1}, \dots ,g_{n}} a_{{{k_1}}  \ldots  {k_n}} \ (A_1 {\overline{\boldsymbol{e}}_{{k_1}}^{1} ) \otimes \dots \otimes ( A_n \overline{\boldsymbol{e}}_{{k_n}}^{n}} ). \] 
We consider the residue class of $t_{H, \mathcal{B}}(I)$ by the action $\rho$ for ${(\mathbb{Z}_{{\Delta}_I})}^{g_{1}} \otimes \dots \otimes {(\mathbb{Z}_{{\Delta}_I})}^{g_{n}}$ and denote it by $t_H(I)$. In fact it is independent of $\mathcal{B}$ as follows.

\begin{proposition}\label{general invariant}
Let $H$ be an $n$-component handlebody-link.
Then $t_H(I)$ is independent of a basis $\mathcal{B}$ of $H_1(H, \mathbb{Z})$ and an HL-homotopy invariant. 
\end{proposition}

\begin{proof}
The proof is by induction on the length $m$ of a sequence $I$.
We will show that for any sequence $I=i_1i_2 \ldots i_m$, 
the following two statements hold.
\begin{itemize}
\item[(A)] ${\Delta}_I$ is independent of the choice of a basis $\mathcal{B}$ of the first homology group $H_1(H, \mathbb{Z})$.
\item[(B)] Fix $i$ appearing in $I$ and $l$ ($1 \leq l \leq g_i$). 
\begin{itemize}
\item[(B-1)] For each $k_1, \dots , k_{i-1}, k_{i+1}, \dots , k_{n}$,  
$$ \overline{\mu}_{e_{k_1}^{1} \cup \dots \cup (-e_{l}^{i}) \cup \dots \cup e_{k_n}^{n} }(I)\equiv  - \overline{\mu}_{e_{k_1}^{1} \cup \dots \cup e_{l}^{i} \cup \dots \cup e_{k_n}^{n} }(I) \mod {\Delta}_I. $$
\item[(B-2)] For each $k_1, \dots , k_{i-1}, k_{i+1}, \dots, k_{n}$ and some $h$ ($1 \leq h \leq g_i$, $h \neq l$),
$$\overline{\mu}_{e_{k_1}^{1} \cup \dots \cup (e_{l}^{i} \sharp_b e_{h}^{i}) \cup \dots \cup e_{k_n}^{n} }(I) 
\equiv \overline{\mu}_{e_{k_1}^{1} \cup \dots \cup e_{l}^{i} \cup \dots \cup e_{n}^{n} }(I) 
+ \overline{\mu}_{e_{k_1}^{1} \cup \dots \cup e_{h}^{i} \cup \dots \cup e_{k_n}^{n} }(I) \mod {\Delta}_I,$$
where $e_{l}^{i} \sharp_b e_{h}^{i}$ is a band sum of $e_{l}^{i}$ and $e_{h}^{i}$ which represents $e_{l}^{i} + e_{h}^{i}$ in $H_1(L_{i}, \mathbb{Z})$.
\end{itemize}
\end{itemize}
For the case $m=2$, ${\Delta}_I=0$ for any choice of a basis of the first homology group.
Then, Lemma \ref{Milnor'sLemma} (2) and \ref{add} make it obvious that the statements (A)  and (B) hold.
Assume that the statements (A) and (B) hold for less than or equal to $m-1$, then we will prove it for $m$.
We consider the change of a basis and denote the indeterminacy with respect to the original basis by ${\Delta}_I$ and the new one by ${\Delta}'_I$.
First of all, we consider the statement (A) for the following three elementary changes of a basis (i), (ii) and (iii): 

\begin{itemize}
\item[(i)]  Exchanging $e^{i}_{l}$ and $e^{i}_{h}$ for some $l, h$ ($1 \leq l, h \leq g_i$  and $l \neq h$). 
\item[(ii)] Changing $e^{i}_{l}$ to $-e^{i}_{l}$ for some $l$ ($1 \leq l \leq g_i$).\
\item[(iii)] Changing $e^{i}_{l}$ to $e^{i}_{l}+e^{i}_{h}$ for some $l, h$ ($1 \leq l, h \leq g_i$  and $l \neq h$). 
\end{itemize}

We consider the case (i). 
Then it is obvious that ${\Delta}_I={\Delta}'_I$.
We consider the case (ii) and (iii). 
For any sequence $I$ of length $m$, we consider its subsequence $J$ of length less than or equal to $m-1$.
If $J$ contains $i$, by the statements (A) and (B) for less than or equal to $m-1$, ${\Delta}_J ={\Delta}'_J$ and then  
$$   \overline{\mu}_{e_{k_1}^{1} \cup \dots \cup (-e_{l}^{i}) \cup \dots \cup e_{k_n}^{n} }(J) \equiv -\overline{\mu}_{e_{k_1}^{1} \cup \dots \cup e_{l}^{i} \cup \dots \cup e_{k_n}^{n} }(J) \mod {\Delta}_J $$
and 
$$\overline{\mu}_{e_{k_1}^{1} \cup \dots \cup (e_{l}^{i} \sharp_b e_{h}^{i}) \cup \dots \cup e_{k_n}^{n} }(J) 
\equiv \overline{\mu}_{e_{k_1}^{1} \cup \dots \cup e_{l}^{i} \cup \dots \cup e_{k_n}^{n} }(J) 
+ \overline{\mu}_{e_{k_1}^{1} \cup \dots \cup e_{h}^{i} \cup \dots \cup e_{k_n}^{n} }(J) \mod {\Delta}_J,$$
for any $k_1, \dots, k_{i-1}, k_{i+1}, \dots, k_n$.
Thus, $\overline{\mu}_{e_{k_1}^{1} \cup \dots \cup (-e_{l}^{i}) \cup \dots \cup e_{k_n}^{n} }(J)$ and 
$\overline{\mu}_{e_{k_1}^{1} \cup \dots \cup (e_{l}^{i} \sharp_b e_{h}^{i}) \cup \dots \cup e_{k_n}^{n} }(J)$ can be divided by ${\Delta}_I$, so ${\Delta'}_I$ can be divided by ${\Delta}_I$.
Similarly $\overline{\mu}_{e_{k_1}^{1} \cup \dots \cup e_{l}^{i} \cup \dots \cup e_{k_n}^{n} }(J)$ and $\overline{\mu}_{e_{k_1}^{1} \cup \dots \cup e_{h}^{i} \cup \dots \cup e_{k_n}^{n} }(J)$ can be divided by ${\Delta}'_I$.
Therefore ${\Delta}_I={\Delta}'_I$ and the statement (A) holds for $I$.
We then obtain the formulas (B-1) and (B-2) for $I$, by Lemma \ref{Milnor'sLemma} (2) and \ref{add}, respectively.
\par
We denote the original basis by $\mathcal{B}$ and the new one by $\mathcal{B}'$, which is obtained by an elementary change (i), (ii) or (iii) of $\mathcal{B}$.
By the statements (B-1) and (B-2), we obtain that 
$$(E, \dots , A_i, \dots, E) \cdot t_{H,\mathcal{B}}(I) = t_{H,\mathcal{B}'}(I), $$ 
where $E$ is the identity matrix and $A_i$ is an elementary matrix which is obtained from the identity matrix by the swapping row $l$ and row $h$ for (i), by changing 1 in the $l$-th position to $-1$  for (ii) and by adding 1 in the $(l,h)$ position for (iii), respectively. 
It is obvious that $t_H(I)$ is HL-homotopy invariant and the proof is complete.
\end{proof}

\begin{corollary}\label{separable2}
An n-component handlebody-link $H$ is HL-homotopic to the trivial handlebody-link if and only if 
$t_{H}(I)=0$ for any $I$. 
\end{corollary}

\begin{proof}
We use claspers. See Section \ref{sec6}.
\end{proof}

\begin{remark} 
We can say that $t_H(I)$ is identified with the following definition essentially, by a property of Milnor's invarinats.
For any sequence $I=i_1 i_2 \ldots i_m$ ($m \leq n$) with distinct indices in $\{1,2, \dots, n\}$,   
 \[ t_H(I) := \sum_{{k_{1}}, \dots ,{k_{m}}=1}^{g_{i_1},  \dots, g_{i_m}}  {\mu}_{e_{k_{1}}^{i_1} \cup  \dots \cup e_{k_{m}}^{i_m} }(I)
  \ \overline{\boldsymbol{e}}_{k_{1}}^{i_1} \otimes \dots \otimes \overline{\boldsymbol{e}}_{k_{m}}^{i_m} \in {(\mathbb{Z}_{{\Delta}_I})}^{g_{i_1}} \otimes  \dots \otimes {(\mathbb{Z}_{{\Delta}_I})}^{g_{i_m}}.\]
\end{remark}

\begin{example}
Let $H$ be a handlebody-link as illustrated in Figure~\ref{exm3}. Let $I=123$.
Then, ${\Delta}_I=2$ and 
$$t_H(I)=1 \ \overline{\boldsymbol{e}}_{{1}}^{1} \otimes  \overline{\boldsymbol{e}}_{{1}}^{2} \otimes \overline{\boldsymbol{e}}_{{1}}^{3}
+ 1 \ \overline{\boldsymbol{e}}_{{2}}^{1} \otimes  \overline{\boldsymbol{e}}_{{2}}^{2} \otimes \overline{\boldsymbol{e}}_{{2}}^{3}
 \in {(\mathbb{Z}_{2})}^{2} \otimes {(\mathbb{Z}_{2})}^{2} \otimes  {(\mathbb{Z}_{2})}^{2}.$$ 
\end{example}
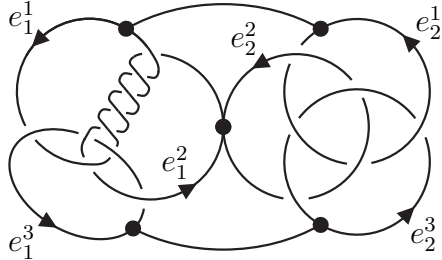
\begin{figure}[ht]
\raisebox{-15 pt}{\begin{overpic}[bb=0 0 189 111, width=160
pt]{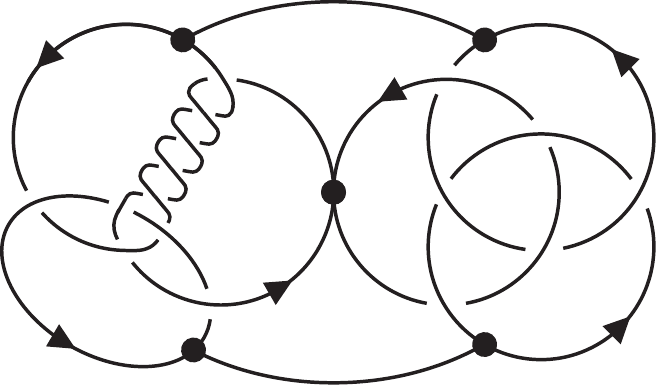}
\put(0,85){$e^1_1$}\put(154,83){$e^1_2$}
\put(58,30){$e^2_1$}\put(84,77){$e^2_2$}
\put(0,0){$e^3_1$}\put(152,3){$e^3_2$}
\end{overpic}}
\caption{Handlebody-link $H$.} \label{exm3}
\end{figure}

\begin{remark}
T.~Fleming defined a numerical invariant $\lambda_{\Phi}(H)$ of a pair of a spatial graph $\Phi$ and its subgraph $H$ under component homotopy in \cite{F}.
Now, we define $\Phi$ as a handlebody-link instead of a spatial graph and $H$ as its component instead of a subgraph.
We then can naturally extend this invariant to a pair of a handlebody-link and its component under HL-homotopy.
Then, the value of $\lambda_{\Phi}(H)$ is the length of first non-vanishing for $t_\Phi(I)$ such that $I$ contains the component number of $H$.
\end{remark}

\subsection{Almost trivial case}
An $n$-component handlebody-link $H$ is called \textit{almost trivial} if its any $(n-1)$-component subhandlebody-link is HL-homotopic to the trivial handlebody-link. From Corollary \ref{separable2}, it is hold that $H$ is almost trivial if and only if $t_H(I)=0$ for any $I$ whose length is less than $n$. We give a bijection between the set of HL-homotopy classes of almost trivial handlebody-links and a union of direct sum of tensor product spaces modulo the action $\rho$.
\par 
Let $\overline{W}_n$ be the set of HL-homotopy classes of $n$-component almost trivial handlebody-links. 
\par
Let $\overline{X}_{m_1, \dots, m_n}$ be $({\mathbb{Z}}^{m_{1}} \otimes \dots \otimes {\mathbb{Z}}^{m_{n}})^{\oplus (n-2)!}$ modulo the diagonal action of $\rho$ in Subsection \ref{mapdef}, 
which acts $\rho$ to every $(n-2)!$ tensor product spaces ${\mathbb{Z}}^{m_{1}} \otimes \dots \otimes {\mathbb{Z}}^{m_{n}}$ at the same time. Then we put 
$$\overline{X}_n = \coprod_{m_1, \dots, m_n \in \mathbb{N}} \overline{X}_{m_1, \dots, m_n}.$$
There is a bijection between $\overline{W}_n$ and $\overline{X}_n$. We consider the lexicographic order on permutations of $\{2, 3,  \dots, n-1\}$.

\begin{theorem} \label{thm2}
 Let $\sigma_p$ $(1\leq p \leq (n-2)!)$ be the $p$-th permutation of $\{2, 3, \dots, n-1\}$ and put $I_p=1\sigma_p(2)\dots\sigma_p(n-1)n$. The following map is well-defined and bijective:
$$\varphi: \overline{W}_n\rightarrow \overline{X}_n; H\mapsto (t_H({I_1}), \dots, t_H({I_{(n-2)!}})).$$
Note that, let $\overline{W}_{g_1, \dots, g_n} \subset \overline{W}_n$ be the set of HL-homotopy classes of $n$-component almost trivial handlebody-links whose genera are $g_1, \dots, g_n$, then the restriction of $\varphi$ to $\overline{W}_{g_1, \dots, g_n}$ is a bijection
$$\varphi|_{\overline{W}_{g_1, \dots, g_n}}: \overline{W}_{g_1, \dots, g_n} \rightarrow \overline{X}_{g_1, \dots, g_n}.$$
\end{theorem}
\begin{proof}
See Section \ref{sec6}.
\end{proof}

\subsection{HL-homotopy invariant for handlebody-link} \label{invs}

\par
From Remark \ref{THid}, $t_{H,\mathcal{B}}(I)$ corresponds to a hypermatix $M_{H,\mathcal{B}}(I)= (\overline{\mu}_{e_{{k_1}}^{1} \cup  \dots \cup e_{{k_n}}^{n} })_{{k_1}, \dots ,{k_n}=1}^{g_{1}, \dots, g_{n}}$ in $M(g_1, \dots, g_n; \mathbb{Z}_\delta)$. Then
we can describe Theorem \ref{thm2} by using hypermatrices. 
Let $\overline{HM}_n$ be the set of $(n-2)!$-tuples of $n$-hypermatrices $(M_1, \dots, M_{(n-2)!})$ whose entries are in $\mathbb{Z}$, where every $M_i$'s are the same size, modulo the elementary transformations which are applied at the same time for every $M_i$. When $n=3$, $\overline{HM}_3$ is the set of 3-hypermatrices up to the elementary transformations. Then, there is a bijection between $\overline{W}_n$ and $\overline{HM}_n$: 
$$\varphi: \overline{W}_n\rightarrow \overline{HM}_n; H\mapsto (M_H({I_1}), \dots, M_H({I_{(n-2)!}})),$$
where $(M_H({I_1}), \dots, M_H({I_{(n-2)!}}))$ is the equivalence class of $(M_{H,\mathcal{B}}(I_1), \dots, M_{H,\mathcal{B}}({I_{(n-2)!}}))$ in $\overline{HM}_n$.

\par
For any hypermatrix $A \in M(m_1, \dots ,m_d)$, 
there exist $r \in \mathbb{N}$ and $\mbox{\boldmath $a$}_{ij} \in \mathbb{Z}^{m_i}$ ($i=1, \dots,  d$) such that
$$ A=\sum_{j=1}^{r}  \mbox{\boldmath $a$}_{1j} \otimes \mbox{\boldmath $a$}_{2j} \otimes \dots \otimes \mbox{\boldmath $a$}_{d j}.$$
The {\it tensor rank} of a hypermatrix $A$ is defined to be the minimum number of $r$ and denoted by ${\rm rank_t}(A)$. 
A hypermatrix has tensor rank zero if and only if it is zero.

The {\it $k$-th flattening map} on $M({m_1, \dots , m_d})$ is the map $f_k$ from a hypermatrix to a usual matrix as follows. The map
$$f_k : M({m_1, \dots , m_d}) \to M({m_k, m_1 \times \dots \times \hat{m}_k \times \dots \times m_d})$$
is defined by 
$$ (f_k(A))_{ij}= (A)_{s_k(i,j)},$$
where $s_k(i,j)$ is the $j$-th element in lexicographic order in the subset of $\langle m_1 \rangle \times \dots \times \langle m_d \rangle$ consisting of elements such that the $k$-th coordinate is the index $i$, and a caret means that the respective entry is omitted.
For example, given 

\[
  A = \left(
    \begin{array}{ccc}
      a_{111} & a_{121}  & a_{131} \\
      a_{211} & a_{221} &  a_{231} \\
      a_{311} & a_{321} &  a_{331} \\
      a_{411} & a_{421} &  a_{431} \\
    \end{array}
  \right| 
  \left.
\begin{array}{ccc}
      a_{112} & a_{122}  & a_{132} \\
      a_{212} & a_{222} &  a_{232} \\
      a_{312} & a_{322} &  a_{332} \\
      a_{412} & a_{422} &  a_{432} \\    
      \end{array}
  \right) \in M(4, 3, 2),
  \]

we obtain that

\[
  f_1(A) = \left(
    \begin{array}{cccccc}
      a_{111} & a_{112} & a_{121}  & a_{122} &  a_{131} & a_{132} \\
      a_{211} & a_{212} & a_{221} &  a_{222} & a_{231} &  a_{232} \\
      a_{311} & a_{312} & a_{321} &  a_{322} & a_{331} &  a_{332} \\
      a_{411} & a_{412} & a_{421} &  a_{422} & a_{431} & a_{432} \\ 
    \end{array}
  \right) \in M(4,6),
\]

\[
  f_2(A) = \left(
    \begin{array}{cccccccc}
      a_{111} & a_{112} & a_{211}  & a_{212} &  a_{311} & a_{312} &a_{411} & a_{412} \\
      a_{121} & a_{122} & a_{221} &  a_{222} & a_{321} &  a_{322} & a_{421} &  a_{422} \\
      a_{131} & a_{132} & a_{231} &  a_{232} & a_{331} &  a_{332} & a_{431} & a_{432} \\
    \end{array}
  \right) \in M(3,8),\]

\[
  f_3(A) = \left(
    \begin{array}{cccccccccccc}
      a_{111} & a_{121} & a_{131}  & a_{211} &  a_{221} & a_{231} &a_{311} & a_{321} & a_{331}  & a_{411} & a_{421} & a_{431} \\
      a_{112} & a_{122} & a_{132} &  a_{212} & a_{222} &  a_{232} & a_{312} &  a_{322} &  a_{332} & a_{412} & a_{422} & a_{432} \\
    \end{array}
  \right) \in M(2,12).
\]

Let $A$ be a hypermatrix in $M(m_1, \dots , m_d)$.
The {\it multilinear rank} of $A$ is defined as a sequence $( {\rm rank} (f_k(A)) )_{k=1,2, \dots, d}$ and denoted by ${\rm rank_m}(A)$. 
An {\it elementary divisor of} $A$ {\it by flattening map} is defined as a sequence of the multiset of elementary devisors of matrices $f_k(A)$'s for $k=1,2, \dots, d$ and denoted by ${\rm ed}(A)$.

The {\it combinatorial hyperdeterminant} of a cubical $d$-hypermatrix $A=(a_{i_1i_2 \ldots i_d}) \in M(m, \dots ,m)$ is defined as 
$$ {\rm det}(A)= \dfrac1{m!} \sum_{\sigma_1, \dots, \sigma_d \in S_m} {\rm sgn} (\sigma_1) \dots {\rm sgn}  (\sigma_d) \prod_{i=1}^{m} a_{\sigma_1(i) \ldots \sigma_d(i)},$$
where $S_m$ is a permutation group.
It is known that for an odd dimension $d$ ($>1$), the combinatorial hyperdeterminant of a cubical $d$-hypermatrix is identically zero.

We give a corollary for Theorem~\ref{thm2} by using invariants of hypermatrices under the transformations (TI), (TII) and (TIII) in Section \ref{sec4}.   

\begin{corollary}\label{rank}
For any handlebody-link $H$, the tensor rank of $M_H(I)$, and the multilinear rank and elementary devisor of $M_H(I)$ by the flattening map are invariant under HL-homotopy.
Moreovere, if the components of $H$ have the same genus, the absolute value of combinatorial hyperdeterminant of $M_H(I)$ is invariant under HL-homotopy.
\end{corollary}

\begin{proof}
The tensor rank of a hypermatrix, and each multilinear rank and elementary devisor of matrix derived from a hypermatrix by the flattening map do not change under elementary transformations (TI), (TII) and (TIII) for hypermatrices.
Moreover, the combinatorial hyperdeterminant of a hypermatrix does not change under elementary transformations (TI) and (TII) for hypermatrices, and changes under elementary transformation (TIII) only sign.
\end{proof}

\subsection{Examples}

\begin{example}
Let $H_1$ and $H_2$ be two handlebody-links depicted in Figure
\ref{exm01}. Then,
$$M_{H_1}(I)=\left(\!\begin{array}{ccc|ccc} 1 & 1 & 1 & 2 & 2 & 2 \\ 0 &
0 & 0 & 0 & 0 & 0\end{array}\!\right), \hspace{0.5cm}
M_{H_2}(I)=\left(\!\begin{array}{ccc|ccc} 1 & 1 & 0 & 1 & 1 & 0 \\ 1 & 1
& 0 & 1 & 1 & 0\end{array}\!\right),$$
where $I=123$. Here, $M_{H_1}(I)$ is transformed to $M_{H_2}(I)$ by elementary transformations. Therefore $H_1$ and $H_2$ are
HL-homotopic.
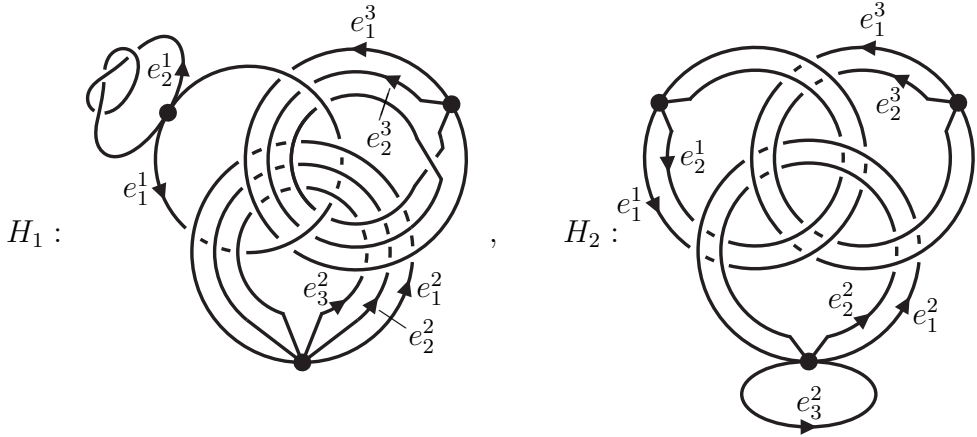
\begin{figure}[ht]
$$
H_1: \hspace{0.2cm}
\raisebox{-49 pt}{\begin{overpic}[bb=0 0 112 99, width=144
pt]{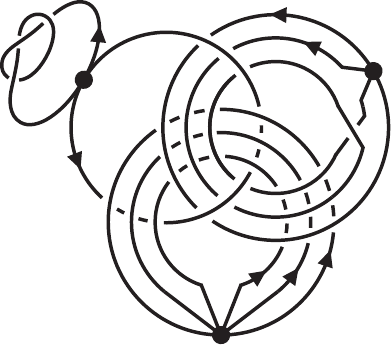}
\put(14, 66){$e^1_1$}\put(23,111){$e^1_2$}
\put(125,28){$e^2_1$} \put(122,10){$e^2_2$}
\put(111,23){\line(3,-2){10}} \put(82,28){$e^2_3$}
\put(100,129){$e^3_1$}\put(106,85){$e^3_2$}
\put(111,95){\line(1,3){4}}
\end{overpic}}
\hspace{0.3cm},\hspace{0.8cm}
H_2: \hspace{0.2cm}
\raisebox{-73 pt}{\begin{overpic}[bb=0 0 97 116, width=125
pt]{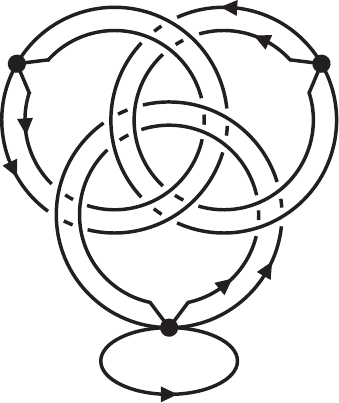}
\put(-10, 84){$e^1_1$}\put(14,102){$e^1_2$}
\put(102,42){$e^2_1$}\put(70,48){$e^2_2$} \put(58,9){$e^2_3$}
\put(82,154){$e^3_1$} \put(88,122){$e^3_2$}
\end{overpic}}
$$
\caption{Handlebody-links $H_1$ and $H_2$.} \label{exm01}
\end{figure}
\end{example}

\begin{example} Let $H_3$ and $H_4$ be two handlebody-links depicted
in Figure \ref{exm02}. Then, 
$$M_{H_3}(I)=\left(\!\begin{array}{ccc|ccc} 1 & 1 & 0 & 0 & 0 & 1 \\ 1 &
1 & 0 & 0 & 0 & 1\end{array}\!\right), \hspace{0.5cm}
M_{H_4}(I)=\left(\!\begin{array}{ccc|ccc} 2 & 0 & 0 & 0 & 1 & 0 \\ 2 & 0
& 0 & 0 & 1 & 0\end{array}\!\right),$$
where $I=123$.
$$M_{H_3}(I)=\left(\!\begin{array}{c} 1\\ 1\end{array}\!\right) \otimes
\left(\!\begin{array}{c} 1\\ 1\\ 0\end{array}\!\right) \otimes
\left(\!\begin{array}{c} 1\\ 0\end{array}\!\right)
+\left(\!\begin{array}{c} 1\\ 1\end{array}\!\right) \otimes
\left(\!\begin{array}{c} 0\\ 0\\ 1\end{array}\!\right) \otimes
\left(\!\begin{array}{c} 0\\ 1\end{array}\!\right),$$
$$M_{H_4}(I)=\left(\!\begin{array}{c} 2\\ 2\end{array}\!\right) \otimes
\left(\!\begin{array}{c} 1\\ 0\\ 0\end{array}\!\right) \otimes
\left(\!\begin{array}{c} 1\\ 0\end{array}\!\right)
+\left(\!\begin{array}{c} 1\\ 1\end{array}\!\right) \otimes
\left(\!\begin{array}{c} 0\\ 1\\ 0\end{array}\!\right) \otimes
\left(\!\begin{array}{c} 0\\ 1\end{array}\!\right),$$
and the first slices of them are not integer multiples of the second
slices respectively. Thus ${\rm rank_t}(H_3)={\rm rank_t}(H_4)=2$.
$${\rm rank_m}(M_{H_3}(I))= ({\rm rank}(f_1(M_{H_3}(I))), {\rm
rank}(f_2(M_{H_3}(I))),{\rm rank}(f_3(M_{(H_3}(I))))=(1, 2, 2)$$
and 
$${\rm rank_m}(M_{H_4}(I))= ({\rm rank}(f_1(M_{H_4}(I))), {\rm
rank}(f_2(M_{H_4}(I))),{\rm rank}(f_3(M_{H_4}(I))))=(1, 2, 2).$$
On the other hands, the elementary divisors of them are
$${\rm ed}(M_{H_3}(I))= ({\rm ed}(f_1(M_{H_3}(I))), {\rm
ed}(f_2(M_{H_3}(I))),{\rm ed}(f_3(M_{H_3}(I))))=(\{1\},\{1,1\},\{1,1\}),$$
$${\rm ed}(M_{H_4}(I))= ({\rm ed}(f_1(M_{H_4}(I))), {\rm
ed}(f_2(M_{H_4}(I))),{\rm ed}(f_3(M_{H_4}(I))))=(\{1\},\{1,2\},\{1,2\}).$$
Thus ${\rm ed}(M_{H_3}(I))\neq{\rm ed}(M_{H_4}(I))$, and $H_3$ is not
HL-homotopic to $H_4$.
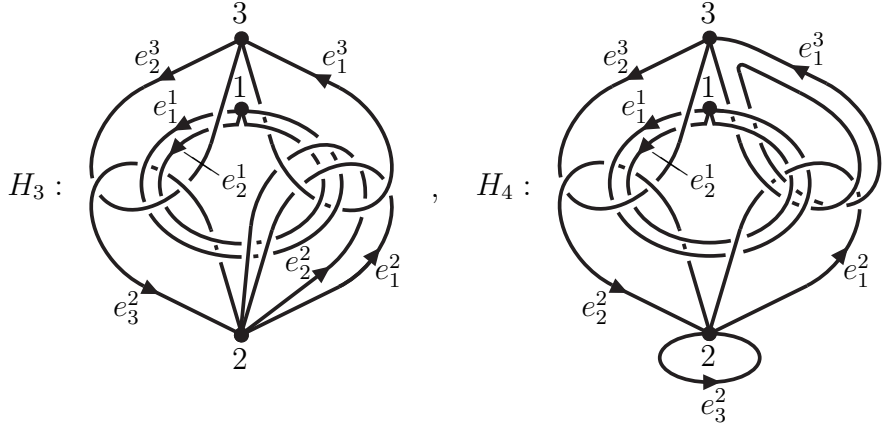
\begin{figure}[ht]
$$
\raisebox{-55 pt}{\begin{overpic}[bb=0 0 160 163, width=115
pt]{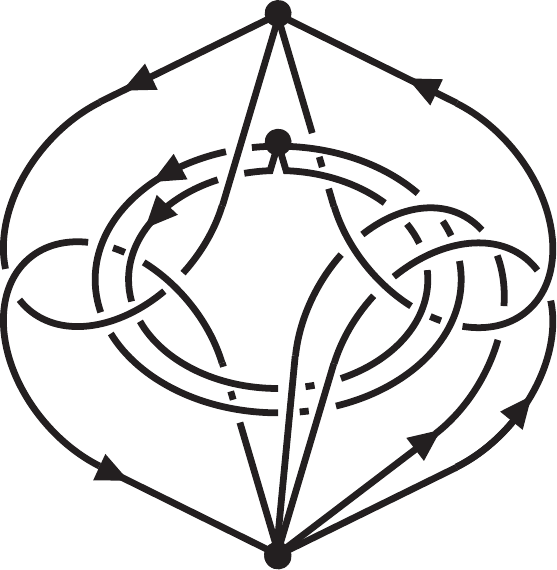}
\put(-31,56){$H_3:$} \put(54,121){3} \put(54,-10){2} \put(54, 93){1}
\put(17,105){$e^3_2$}\put(88,105){$e^3_1$}
\put(24,86){$e^1_1$}\put(50,58){$e^1_2$} \put(36,72){\line(3,-2){13}}
\put(9,10){$e^2_3$}\put(74,27){$e^2_2$} \put(108,24){$e^2_1$}
\end{overpic}}
\hspace{0.5cm},\hspace{1.5cm}
\raisebox{-73 pt}{\begin{overpic}[bb=0 0 170 189, width=122
pt]{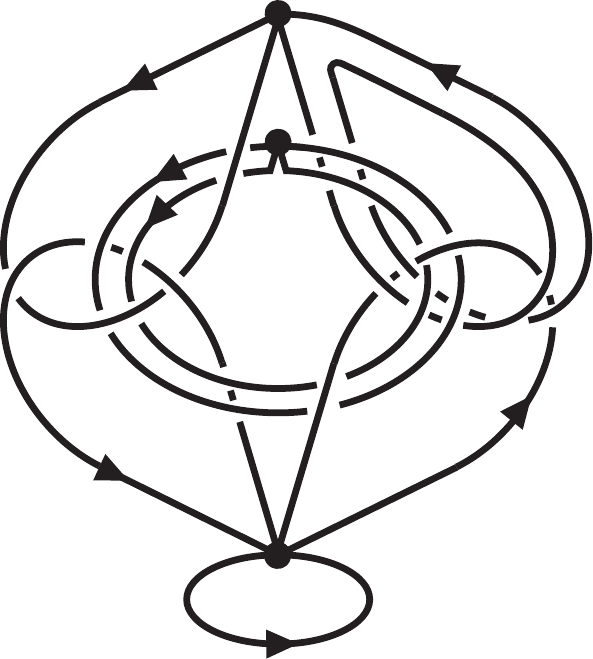}
\put(-31,74){$H_4:$} \put(54,139){3} \put(54,9){2} \put(54, 111){1}
\put(17,123){$e^3_2$}\put(91,126){$e^3_1$}
\put(24,104){$e^1_1$}\put(50,76){$e^1_2$} \put(36,90){\line(3,-2){13}}
\put(9,28){$e^2_2$}\put(54,-9){$e^2_3$} \put(108,42){$e^2_1$}
\end{overpic}}
$$
\caption{Handlebody-links $H_3$ and $H_4$.} \label{exm02}
\end{figure}
\end{example}

\begin{remark}
Similar way to Corollary~\ref{rank}, we can give comparable invariants for $t_H(I)$ of general handlebody-link $H$ by using the tensor rank.
\end{remark}

\section{Proofs of Corollary \ref{separable2} and Theorem \ref{thm2}} \label{sec6}
\par
In this section, we prove Corollary \ref{separable2} and Theorem \ref{thm2}. We use the clasper theory. 

\subsection{Review of clasper theory} \label{sec8}
The clasper theory was introduced by K.~Habiro \cite{Ha}.  
We define a tree clasper and introduce a part of its properties. 
For general definitions and properties, we refer the reader to \cite{Ha}.  

A disk $T$ embedded in $S^3$ is called a {\em tree clasper} 
for a (string) link $L$ if it satisfies the following two conditions:
\begin{enumerate} 
\item The embedded disk $T$ is decomposed into bands and disks, where each band connects two distinct disks and each disk attaches either 1 or 3 bands.
We call a band an {\it edge} and a disk attached 1 band a {\em leaf}. 
\item The embedded disk $T$ intersects the (string) link $L$ transversely so that the intersections are contained in the interiors of the leaves. 
\end{enumerate}
We call a tree clasper $T$ with $k+1$ leaves a \emph{$C_k$-tree}. 
A $C_k$-tree is {\it simple} if each leaf intersects $L$ at exactly one point.  

Given a $C_k$-tree $T$ for a (string) link $L$, there exists a procedure to construct a
framed link in a regular neighborhood of $T$. 
We call surgery along the framed link {\em surgery along} $T$. 
Because there is an orientation-preserving homeomorphism which fixes the boundary, 
from the regular neighborhood $N(T)$ of $T$ to the manifold obtained from $N(T)$ by surgery along $T$, 
we can regard the surgery along $T$ as a local move on $L$. 
We denote by $L_T$ the (string) link obtained from $L$ by surgery along $T$. 
For example, surgery along a $C_k$-tree is a local move as showed in Figure~\ref{Ck-tree}.
Similarly, let $T_1 \cup \dots \cup T_m$ be a disjoint union of tree claspers for $L$, 
then we define $L_{T_1 \cup \dots \cup T_m}$ as the (string) link obtained by surgery along $T_1 \cup \dots \cup T_m$. 

\begin{figure}[ht]
$$\raisebox{-15 pt}{\begin{overpic}[bb=0 0 217 73, width=160
pt]{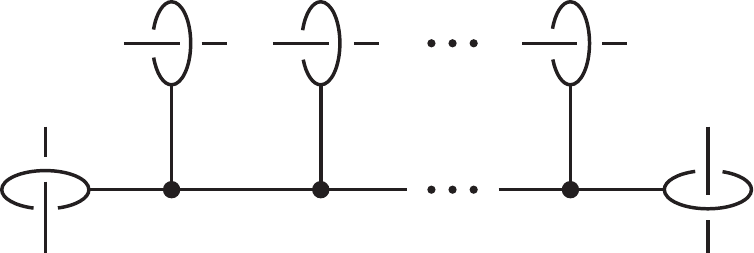}
\end{overpic}}
  \hspace{0.5cm} \mbox{\large$\xrightarrow[]{\rm surgery}$}
\hspace{0.5cm}
\raisebox{-33 pt}{\begin{overpic}[bb=0 0 229 9, width=180
pt]{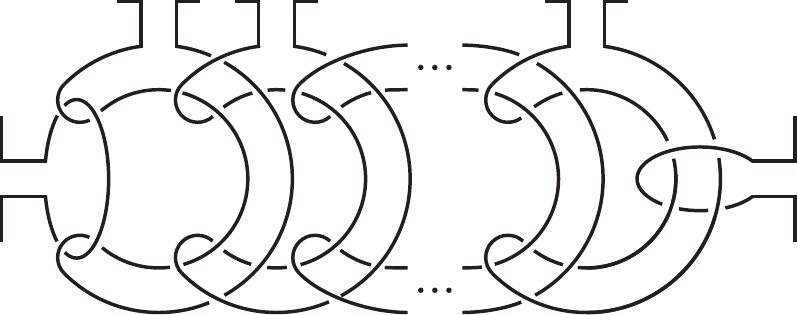}
\end{overpic}}$$
    \caption{$C_k$-tree and local move.}
    \label{Ck-tree}
\end{figure}

The $C_k$-equivalence is an equivalence relation on (string) links generated by surgeries along $C_k$-trees and ambient isotopy.  
By the definition of a $C_k$-tree, it is easy to see that a $C_k$-equivalence implies a $C_{k-1}$-equivalence.
The set of $C_k$-equivalence classes of string links forms a group under the composition.
It is known that a $C_1$-tree corresponds to the crossing change, a $C_2$-tree corresponds to the delta move  in \cite{Mat, MuNa} and a $C_3$-tree corresponds to the  clasp-pass move in \cite{Hab}.

Let $\pi$ be a bijection of $ \{ 1, 2, \dots , n\}$ such that $\pi(1)=1 $ and $\pi(n)=n $,
and let $\Pi$ be the set of such bijections. 
Consider an element of ${\Pi}$ as a sequence.
For any $\pi \in {\Pi}$, let $T_\pi$ and $T_\pi^{-1}$ be simple $C_{n-1}$-trees as illustrated in Figure \ref{C_n-trees},
which are the images of homeomorphisms from the neighborhoods of $T_\pi$ and $T_\pi^{-1}$ to the 3-balls.
Here, $\oplus$ means a positive half-twist.

\begin{figure}[ht]
$$\raisebox{33 pt}{\begin{overpic}[bb=0 0 577 131, 
width=400pt]{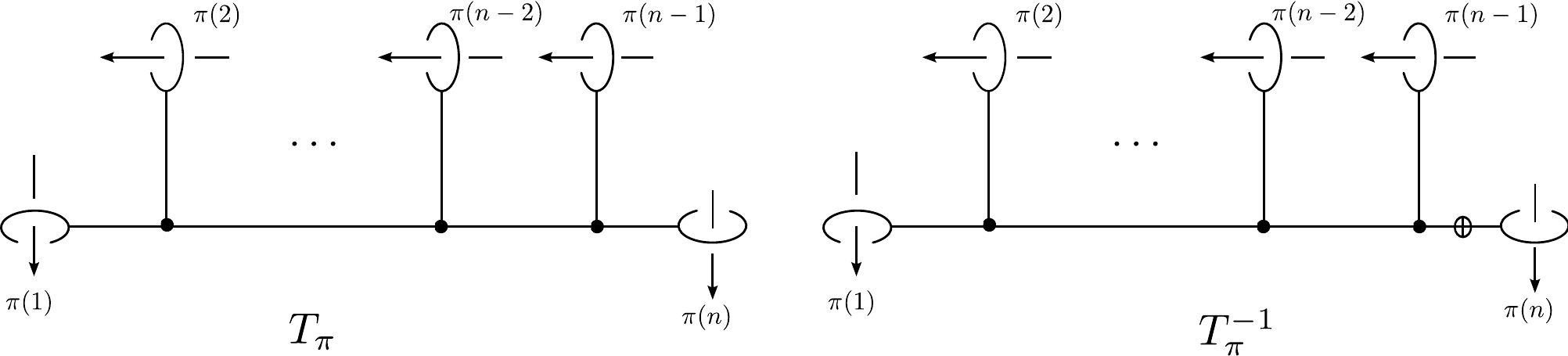}\end{overpic}}$$
\vspace{-1.8cm}
\caption{$C_{n-1}$-trees $T_\pi$ and $T_\pi^{-1}$.} \label{C_n-trees}
\end{figure}

Let ${\bf 1}_n$ be the $n$-component trivial string link. 
Although $({\bf 1}_n)_{T_\pi}$ and $({\bf 1}_n )_{T_\pi^{-1}}$ are not unique up to ambient isotopy, it is unique up to $C_n$-equivalence, by Lemmas~\ref{edgecrossing} and \ref{inverse} below.
Therefore for any $\pi \in \Pi$, we may choose $({\bf 1}_n)_{T_\pi}$ and $({\bf 1}_n)_{T_\pi^{-1}}$ uniquely up to $C_n$-equivalence.
We then have the following lemma by \cite{Mil01} (cf. Figure~\ref{Ck-tree}). 
Here, $\mu$-invariants are link-homotopy invariants for string links (see \cite{HL}). 
For any string link $L$, $\mu_{L}(I)$ coincides with $\bar{\mu}_{\hat{L}}(I)$ modulo $\Delta_{\hat{L}}(I)$, 
where $\hat{L}$ is a link obtained by the closure of $L$.

\begin{lemma} \label{MIlnorproperty} 
For any $\pi,\pi' \in \Pi$,
$$ {\mu}_{({\bf 1}_n )_{T_\pi}}(\pi')= \left\{ 
          \begin{array}{ll}
               1 & \text{ if } \pi=\pi' \\
               0 & \text{ if } \pi \neq \pi',  
          \end{array} 
            \right. 
$$ 
and the Milnor's $\mu$-invariants of $({\bf 1}_n)_{T_\pi}$ of length  less than or equal to $n-1$ vanish.
\end{lemma}

\begin{lemma}[\cite{MY}] \label{additivity} 
Let $L$ and $L'$ be $n$-component string links. 
Let $m$ and $m'$ be integers. 
If $\mu_{L}(I)=0$ for any $I$ with $|I| \leq m$ 
and $\mu_{L'}(I')=0$ for any $I'$ with $|I'| \leq m'$, 
then for any $J$ with $|J| \leq m+m'$
\[ \mu_{L \cdot L'}(J)= \mu_{L}(J) + \mu_{L'}(J). \\ \]  
\end{lemma}

\begin{lemma}[\cite{Ha}]  \label{edgecrossing}
Let $T_1$ be a simple $C_{k}$-tree for a (string) link $L$, and $T'_1$ be obtained from $T_1$ by changing a crossing between an edge of $T_1$ and an edge of another simple tree $T_2$ for $L$ (resp. a component of $L$ or an edge of $T_1$) (see Figure \ref{edge}). 
Then, $L_{T_1 \cup T_2}$ is $C_{k+1}$-equivalent to $L_{{T'_1} \cup {T_2}}$ (resp. $L_{T_1}$ is $C_{k+1}$-equivalent to $L_{T'_1}$),
where $C_{k+1}$-equivalence is realized by simple $C_{k+1}$-trees for $L$. 
\end{lemma}

\begin{figure}[ht]
$$\raisebox{33 pt}{\begin{overpic}[bb=0 0 831 178, 
width=340pt]{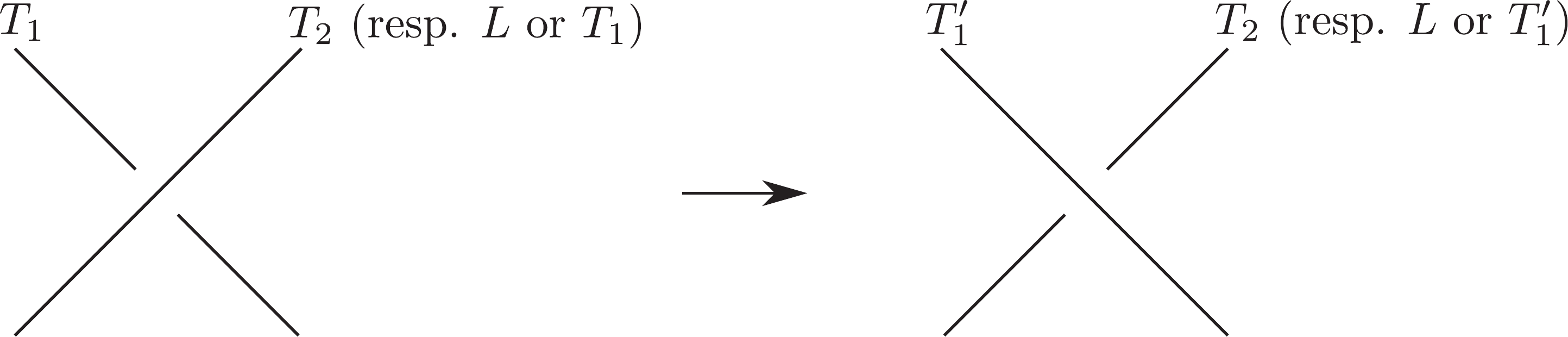}\end{overpic}}$$
\vspace{-1.8cm}
\caption{Edge crossing change.} \label{edge}
\end{figure}

\begin{lemma}[\cite{Ha}] \label{sliding}
Let $T_1$ be a simple $C_{k_1}$-tree for a (string) link $L$ 
and $T_2$ a simple $C_{k_2}$-tree for $L$, where $T_1$ and $T_2$ are disjoint.  
Let ${T'_1}$ be obtained from $T_1$ by sliding a leaf of $T_1$ over a leaf of $T_2$ 
(see Figure~\ref{slide}). 
Then, $L_{T_1 \cup T_2}$ is $C_{k_1+1}$-equivalent to $L_{{T'_1} \cup {T_2}}$,
where $C_{k_1+1}$-equivalence is realized by simple $C_{k_1+1}$-trees for $L$.
\end{lemma}

\begin{figure}[ht]
$$\raisebox{33 pt}{\begin{overpic}[bb=0 0 498 179, 
width=230pt]{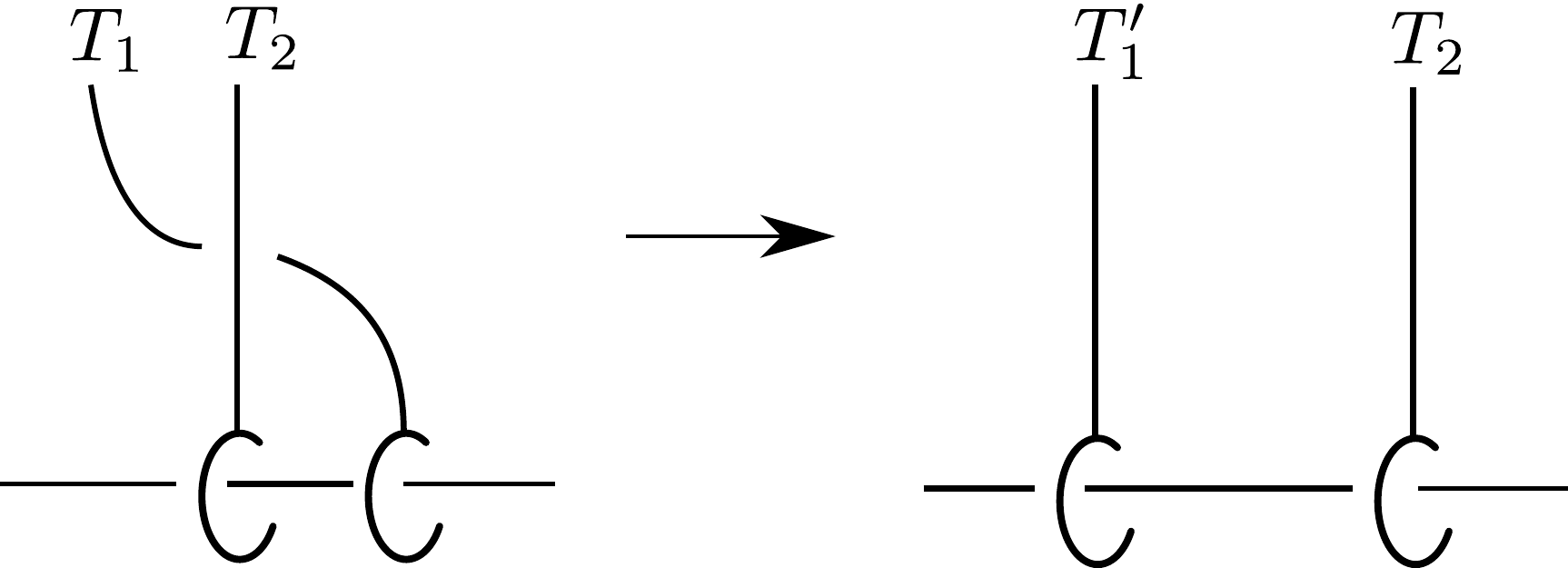}\end{overpic}}$$
\vspace{-1.8cm}
\caption{Leaf slide.} \label{slide}
\end{figure}

\begin{lemma}[\cite{Ha}] \label{inverse}
Let $T$ be a simple $C_k$-tree for the trivial n-component string link ${\bf 1}_n$ and $T'$ a simple $C_k$-tree obtained from $T$ by adding a half-twist on an edge.
Then, $({\bf 1}_n)_T \cdot ({\bf 1}_n)_{T'}$ is ${C_{k+1}}$-equivalent to ${\bf 1}_n$,
where $C_{k+1}$-equivalence is realized by simple $C_{k+1}$-trees for $L$.
\end{lemma}

\begin{lemma}[\cite{Ha}]  \label{IHX}
Let $T_I$, $T_H$ and $T_X$ be simple $C_{k}$-trees for the trivial n-component string link ${\bf 1}_n$ which differ only in a small ball as in illustrated in Figure \ref{IHXmove}. 
Here, $\oplus$ means a positive half-twist. 
Then $({\bf 1}_n)_{T_I}$ is $C_{k+1}$-equivalent to $({\bf 1}_n)_{T_H} \cdot ({\bf 1}_n)_{T_X}$, 
where $C_{k+1}$-equivalence is realized by simple $C_{k+1}$-trees for ${\bf 1}_n$.
\end{lemma}

\begin{figure}[ht]
$$\raisebox{33 pt}{\begin{overpic}[bb=0 0 440 149, 
width=230pt]{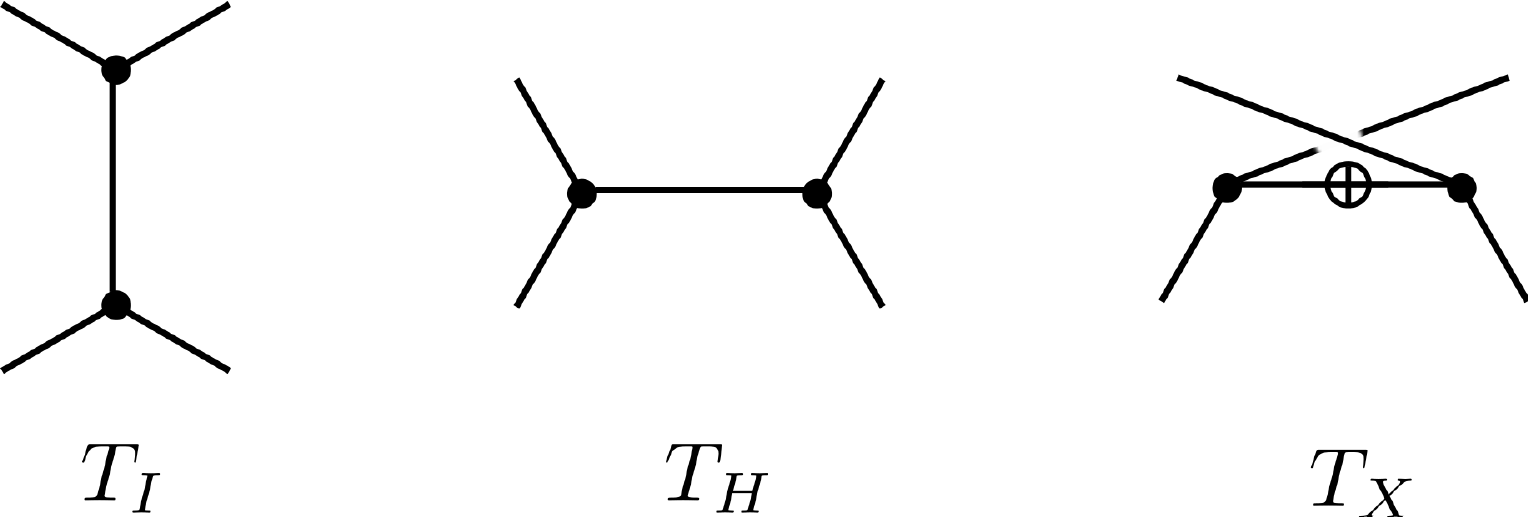}\end{overpic}}$$
\vspace{-1.8cm}
 \caption{IHX move.}
\label{IHXmove}
\end{figure}

\begin{lemma}[\cite{Meil}] \label{sepa}
Let $T$ be a $C_k$-tree for a (string) link $L$. Let $f_1$ and $f_2$ be two disks obtained by splitting a leaf $f$ of $T$ along an arc $a$ as illustrated in figure~\ref{separate} (i.e. $f = f_1 \cup f_2$ and $f_1 \cap f_2 = a$). Then, $L_T$ is ${ C_{k+1}}$-equivalent to $L_{T_1 \cup T_2}$, where $T_i$ denotes a $C_k$-tree for $L$ obtained from a parallel copy of $T$ by replacing $f$ by $f_i$ $(i = 1, 2)$.
\end{lemma}

\begin{figure}[ht]
$$\raisebox{-28 pt}{\begin{overpic}[bb=0 0 116 126, height=80
pt]{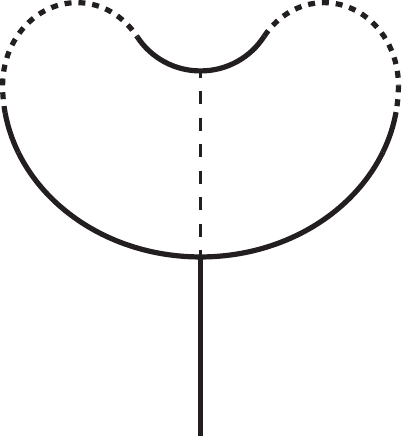}
\put(-12,60){$f$} \put(27,48){$a$} \put(28,-11){ $T$}
\end{overpic}}
  \hspace{0.4cm} \mbox{\large$\longrightarrow$} \hspace{0.6cm}
\raisebox{-28 pt}{\begin{overpic}[bb=0 0 159 126, height=80
pt]{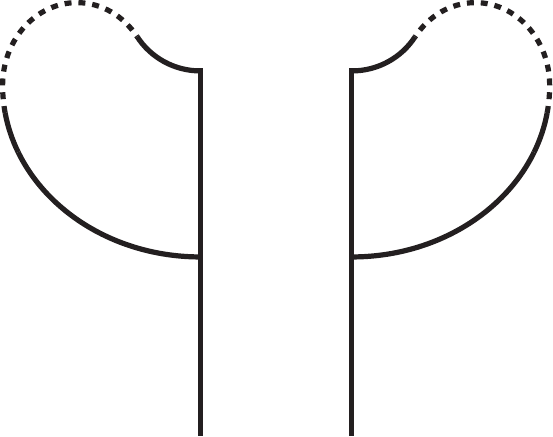}
\put(-13,60){$f_1$} \put(103,60){$f_2$}
\put(39,48){$a$} \put(56,48){$a$}
\put(28,-11){ $T_1$} \put(57,-11){ $T_2$}
\end{overpic}}$$
\caption{Separating of leaf.} \label{separate}
\end{figure}
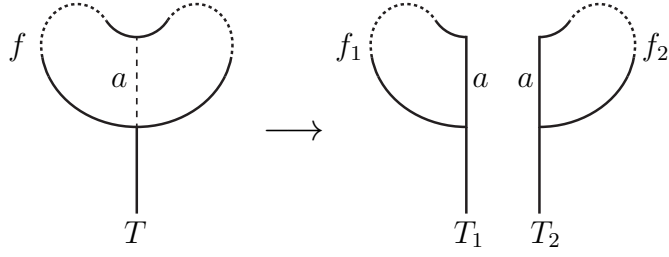

\begin{lemma}[\cite{FY}] \label{multitree}
A tree clasper whose at least 2 leaves attach the same component of a (string) link vanishes up to link-homotopy. 
In particular, a $C_n$-tree for an $n$-component (string) link vanishes up to link-homotopy.
\end{lemma}

\subsection{Proof of Corollary \ref{separable2} and Theorem \ref{thm2}}
We prepare a lemma.
\begin{lemma}\label{HBLtri}
Let $\Gamma_0 = \gamma_1 \cup \dots \cup \gamma_n$ be an oriented plane graph whose components are $n$ bouquet graphs where $\gamma_i$ has $g_i$ loop edges. Let $\Gamma= \eta_1 \cup \dots \cup \eta_n$ be an oriented spacial graph whose component $\eta_i$ is an embedded oriented  bouquet graph with $g_i$ loop edges. If any $(n-1)$-component subgraph of $\Gamma$ is component homotopic to the plane graph, then $\Gamma$ is equivalent to $\Gamma_0$ up to $C_{n-1}$-equivalence and component homotopy. 
\end{lemma}

\begin{proof}
Let $e^i_j$ be the $j$-th edge of $\eta_i$ ($1 \leq j \leq g_i$).
By the assumption of $\Gamma$, for any $I$ with length less than $n$,
$\overline{\mu}_{e^1_{k_1} \cup \dots \cup e^n_{k_n}}(I)=0$
for any $k_1, \dots ,k_n$.
There exists a disjoint union $T_1$ of $C_1$-trees for $\Gamma_0$ such that $\Gamma$ is ambient isotopic to $(\Gamma_0)_{T_1}$, because a $C_1$-tree corresponds to a crossing change.
From the above vanishing condition of $\Gamma$, we can transform these trees of $T_1$ by using Lemma~\ref{edgecrossing}, \ref{sliding} and \ref{IHX} so that all $C_1$-trees with attaching two different components vanish by Lemma \ref{inverse}.
Therefore, by Lemma \ref{multitree}, we obtain a disjoint union $T_2$ of $C_i$-trees ($i \geq 2$) up to component homotopy  and $(\Gamma_0)_{T_1}$ is component homotopic to $(\Gamma_0)_{T_2}$.
Similarly, by the above vanishing condition and Lemma~\ref{edgecrossing}, \ref{sliding}, \ref{inverse}, \ref{IHX} and \ref{multitree}, we can construct a sequence $T_3, T_4, \dots , T_{n-1}$ of disjoint unions of tree claspers so that each tree of $T_i$ is a $C_k$-tree ($k \geq i$) and $(\Gamma_0)_{T_1}$ is component homotopic to $(\Gamma_0)_{T_i}$.
Therefore $\Gamma$ is equivalent to $\Gamma_0$ up to component homotopy and $C_{n-1}$-equivalence.
\end{proof}

We prove Corollary \ref{separable2}.
\begin{proof}[Proof of Corollary \ref{separable2}]
A sufficient condition is trivial by Proposition \ref{general invariant}.
We consider a necessary condition.  
Suppose that $T_H(I)=0$ for any $I$.
Then, for any basis $\mathcal{B}=\{e^1_1, \dots, e^1_{g_1}, \dots, e^n_1, \dots, e^n_{g_n}\}$ of $H_1(H; \mathbb{Z})$, $\overline{\mu}_{e^1_{k_1}\cup \dots \cup e^n_{k_n} }(I)=0$ for any $I$ and  $k_1, \dots ,k_n$. For a fixed basis $\mathcal{B}$, we give an oriented bouquet presentation $\Gamma$ with respect to $\mathcal{B}$. Let $\Gamma_0$ be a plane graph which is a union of $n$ oriented 
bouquet graphs. Similar way to Lemma \ref{HBLtri}, we have a disjoint union $T_n$ of $C_k$-trees ($k \geq n$) and $\Gamma$ is component homotopic to $(\Gamma_0)_{T_n}$. Since $(\Gamma_0)_{T_n}$ is component homotopic $\Gamma_0$, $\Gamma$ is component homotopic to $\Gamma_0$. By Proposition \ref{HandB}, $H$ is HL-homotopic to the trivial handlebody-link.
\end{proof}

\par
To prove Theorem \ref{thm2}, we use $(n-2)!$ kinds of simple $C_{n-1}$-trees for spacial graphs whose shapes are as in Figure \ref{CnCla} left, where \fbox{$\sigma_p$} is the positive braid corresponding to the permutation $\sigma_p$ such that every pair of strings crosses at most one. 
\begin{figure}[ht]
$$ \raisebox{-45 pt}{\begin{overpic}[bb=0 0 226 211, height=100 pt]{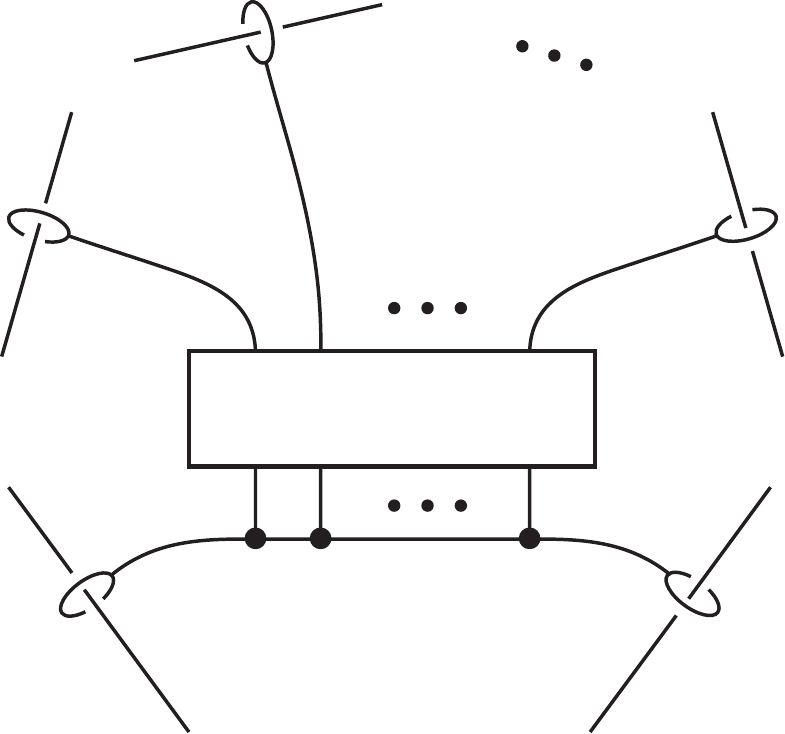}
 \put(-7,24){1} \put(-2,81){2} \put(48,105){$3$} \put(104,80){$n\!-\!1$} \put(107,23){$n$}
 \put(51, 43){$\sigma_p$}
\end{overpic}} 
\hspace{0.8cm}\mbox{\LARGE$\rightarrow$}\hspace{0.8cm}
\raisebox{-37 pt}{\begin{overpic}[bb=0 0 141 129, height=95 pt]{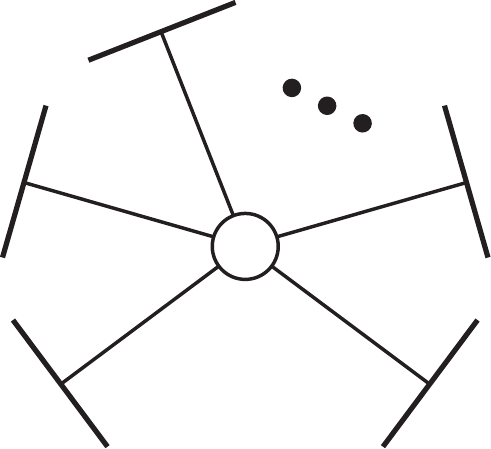}
 \put(-5,18){1} \put(-3,70){2} \put(46,100){$3$} \put(100,68){$n\!-\!1$} \put(104,21){$n$}
 \put(49, 41){$p$}
\end{overpic}}$$
\caption{$C^{(p)}_{n-1}$-tree.} \label{CnCla}
\end{figure}
We call the simple $C_{n-1}$-tree derived from $\sigma_p$ a $C^{(p)}_{n-1}$-tree. By Lemma~\ref{IHX} and other lemmas (Lemma~\ref{sliding}, \ref{inverse} and \ref{multitree}), we can transform any $C_{n-1}$-tree to some $C^{(p)}_{n-1}$-trees up to HL-homotopy. 
For simplicity, we represent the $C^{(p)}_{n-1}$-tree as in Figure \ref{CnCla} right. We also abbreviate parallel $C^{(p)}_{n-1}$-trees as in Figure \ref{ParCnCla}, where a bottom lines is an edges of the $n$-th component, $\oplus$ means a half-twist of an edge and we consider the edge orientations of graphs.
\begin{figure}[ht]
$$
\raisebox{-61 pt}{\begin{overpic}[bb=0 0 242 243, height=117 pt]{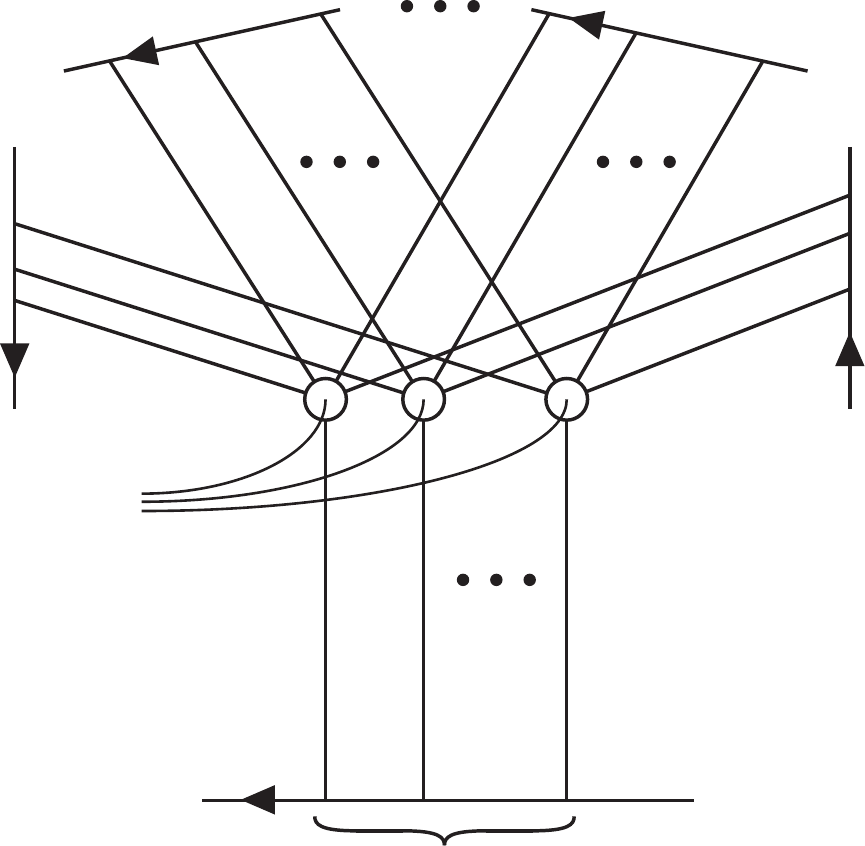}
 \put(-2,47){1} \put(8,113){2} \put(103,113){$n\!-\!2$} \put(106,50){$n\!-\!1$} \put(101, 4){$n$}
 \put(59, -10){$r$}  \put(11, 45){$p$}
\end{overpic}} 
\hspace{0.5cm}\mbox{\LARGE$\rightarrow$}\hspace{0.5cm}
\raisebox{-55 pt}{\begin{overpic}[bb=0 0 222 226, height=110 pt]{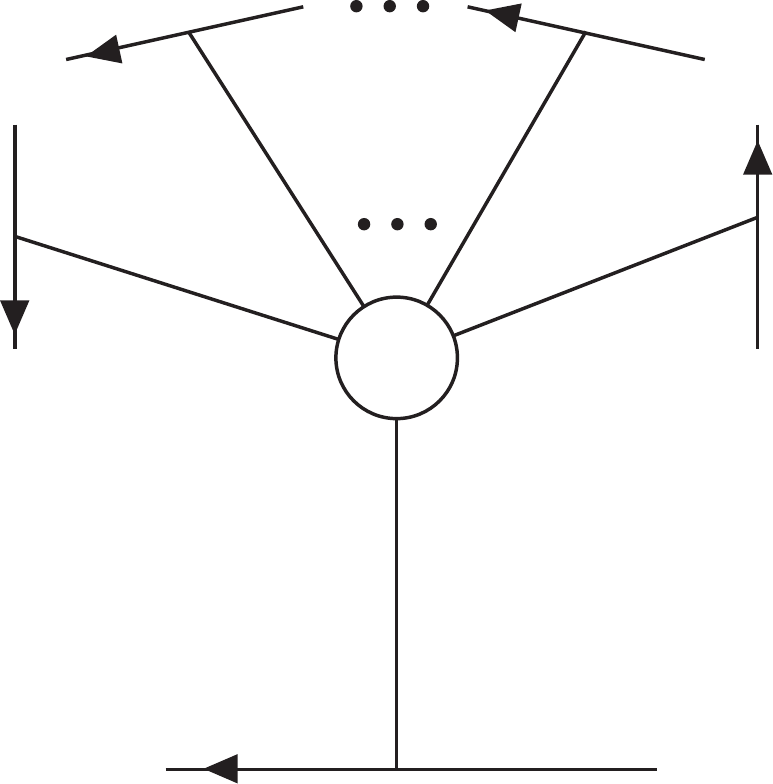}
 \put(-1,46){1} \put(6,108){2} \put(90,108){$n\!-\!2$} \put(94,48){$n\!-\!1$} \put(96,-1){$n$}
 \put(53, 58){$p$}  \put(41, 46){$r$}
\end{overpic}}
$$
\vspace{1.0cm}
$$
\raisebox{-61 pt}{\begin{overpic}[bb=0 0 249 243, height=117 pt]{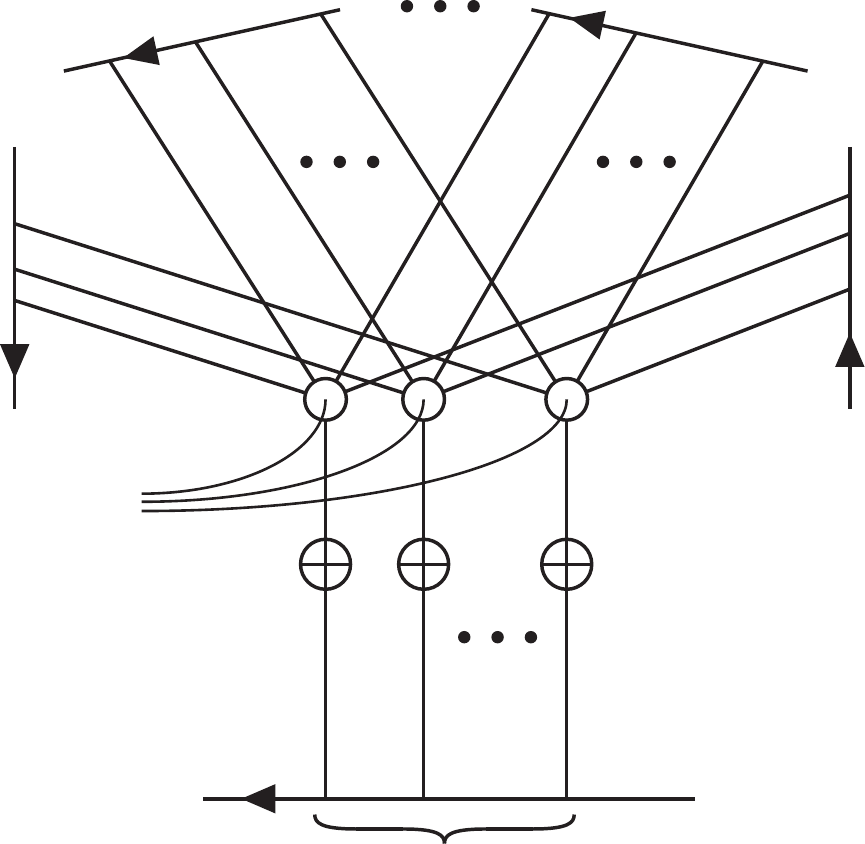}
 \put(-2,47){1} \put(8,113){2} \put(103,113){$n\!-\!2$} \put(106,50){$n\!-\!1$} \put(101, 4){$n$}
 \put(59, -10){$r$}  \put(11, 45){$p$}
\end{overpic}} 
\hspace{0.5cm}\mbox{\LARGE$\rightarrow$}\hspace{0.5cm}
\raisebox{-55 pt}{\begin{overpic}[bb=0 0 222 226, height=110 pt]{n-clasper07-2.pdf}
 \put(-1,46){1} \put(6,108){2} \put(90,108){$n\!-\!2$} \put(94,48){$n\!-\!1$} \put(96,-1){$n$}
 \put(53, 58){$p$} \put(31, 46){$-r$}
\end{overpic}}
$$
\vspace{0.1cm}
\caption{Parallel $C^{(p)}_{n-1}$-trees (upper) and parallel half-twisted $C^{(p)}_{n-1}$-trees (bottom).} \label{ParCnCla}
\end{figure}
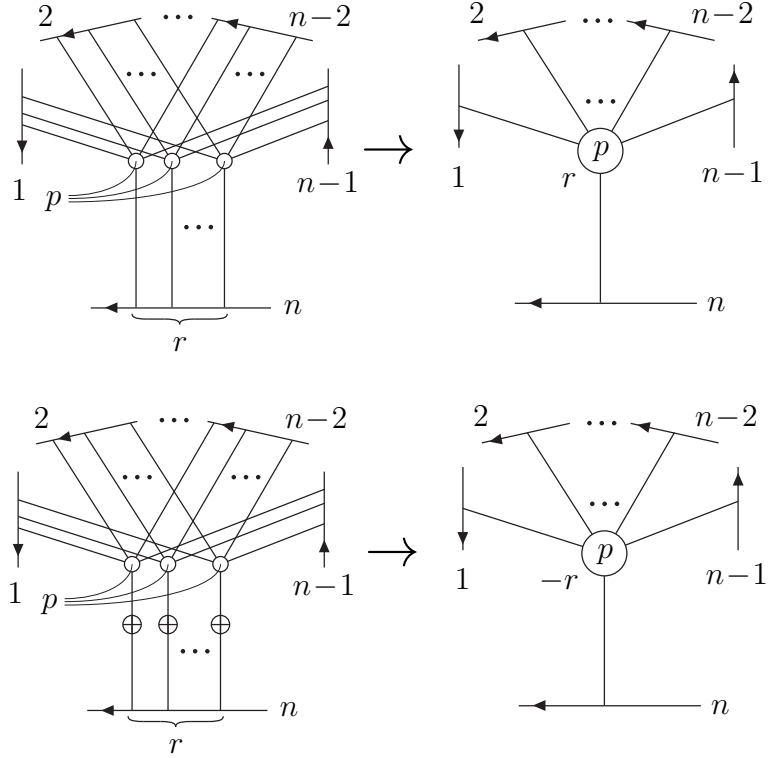
\par
From the lemmas in Subsection \ref{sec8}, the local moves for $C^{(p)}_{n-1}$-trees in Figure \ref{MoCnCla} hold up to a $C_n$-equivalence, where $p$ and $q$ are integers and the horizontal line in (1) is an edge of a graph or a simple $C_{n-1}$-tree. 
The move (1) is to pass an edge trough an edge of a $C^{(p)}_{n-1}$-tree. The move (2) is to exchange of leafs. The move (3) is to move a half twist from an edge to an adjacent edge. The move (4) is to cancel two half twists on the same edge. The move (5) is to cancel of non-twisted and half-twisted $C^{(p)}_{n-1}$-trees in parallel. The move (6) is to erase a $C^{(p)}_{n-1}$-tree whose two leaves are on the same component. In the move (7), by sliding a leaf of a $C^{(p)}_{n-1}$-tree over a trivalent vertex, it changes to two copys of the $C^{(p)}_{n-1}$-tree. Move (1) is derived from Lemma \ref{edgecrossing}, (2) from Lemma \ref{sliding}, (3), (4) and (5) from Lemma \ref{inverse}, (6) from Lemma  \ref{multitree} and (7) from Lemma \ref{sepa}. We remark that the $C^{(p)}_{1}$-tree and the $C^{(p)}_{2}$-tree are equal to the Hopf chord the Borromean chord in \cite{TY}. The moves for the chords are proved schematically there.

\par
\begin{figure}[ht]
$$(1) \hspace{0.2cm} \raisebox{-20 pt}{\begin{overpic}[bb=0 0 142 153, height=50 pt]{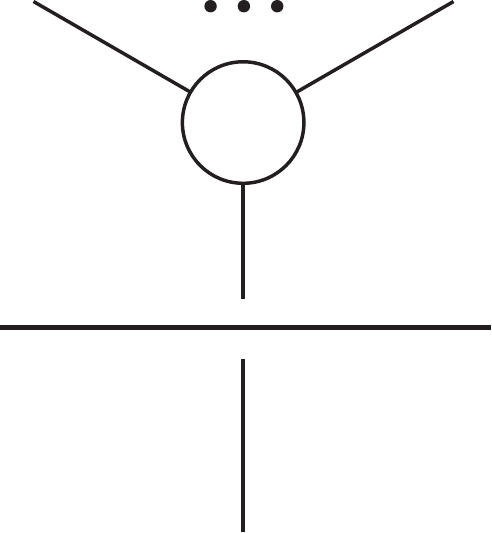}
\put(21,39){$_p$}
\end{overpic}} 
\hspace{0.3cm}\mbox{\LARGE$\leftrightarrow$}\hspace{0.3cm}
\raisebox{-20 pt}{\begin{overpic}[bb=0 0 142 153, height=50 pt]{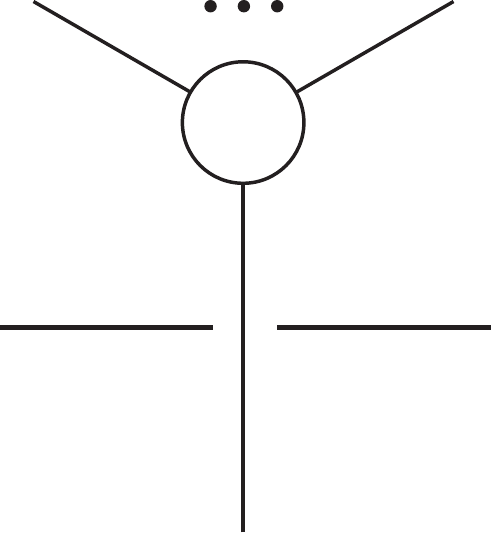}
\put(21,39){$_p$}
\end{overpic}} \hspace{0.3cm}
(2) \hspace{0.2cm} \raisebox{-20 pt}{\begin{overpic}[bb=0 0 313 154, height=50 pt]{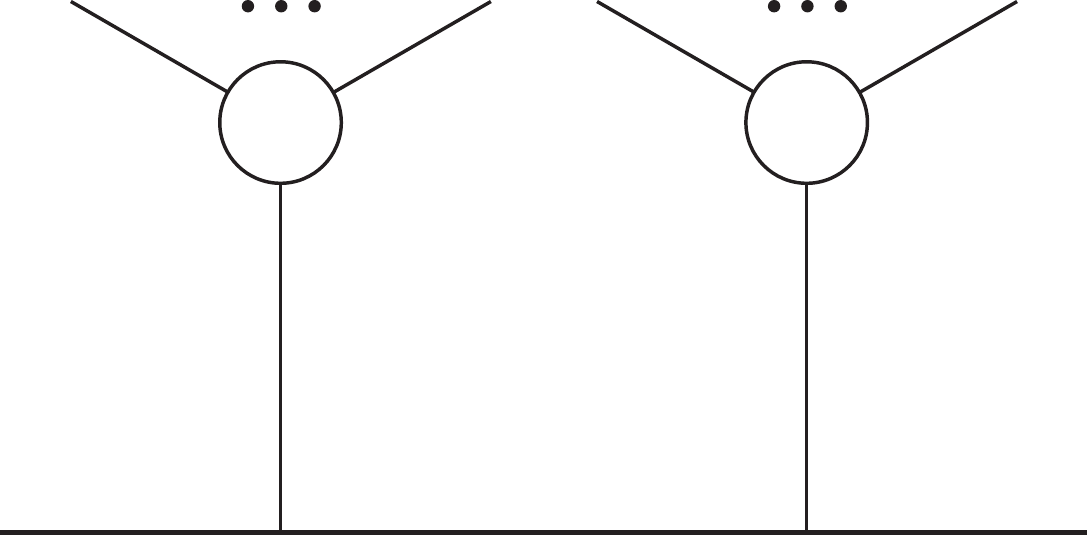}
\put(24,39){$_p$} \put(73,39){$_q$}
\end{overpic}} 
\hspace{0.3cm}\mbox{\LARGE$\leftrightarrow$}\hspace{0.3cm}
\raisebox{-20 pt}{\begin{overpic}[bb=0 0 313 154, height=50 pt]{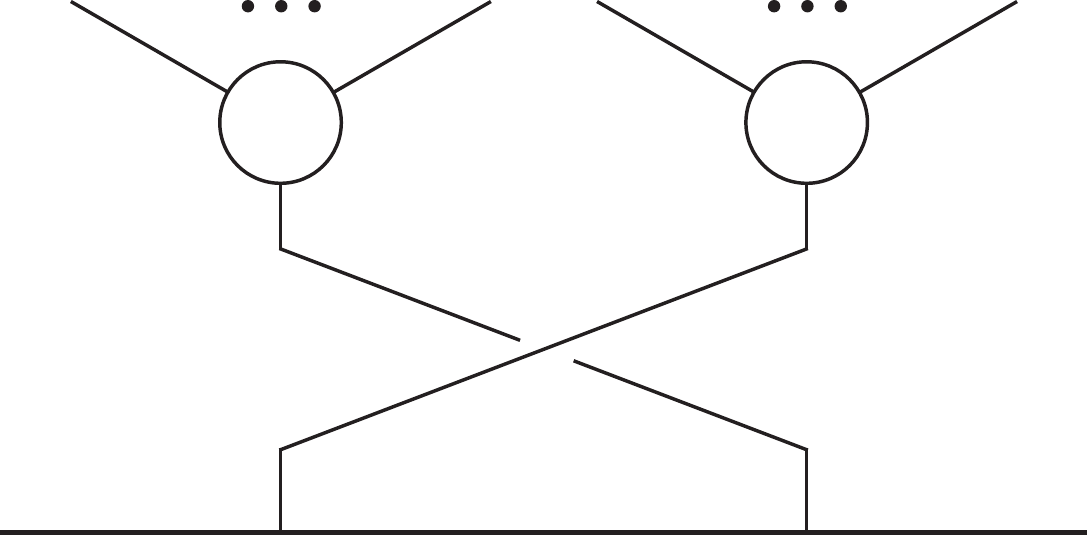}
\put(24,39){$_p$} \put(73,39){$_q$}
\end{overpic}} $$
\vspace{0.3cm}
$$ (3)  \hspace{0.2cm} \raisebox{-20 pt}{\begin{overpic}[bb=0 0 90 90, height=50 pt]{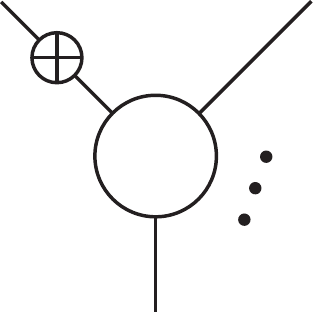}
\put(23,24){$p$}
\end{overpic}} 
\hspace{0.3cm}\mbox{\LARGE$\leftrightarrow$}\hspace{0.3cm}
\raisebox{-20 pt}{\begin{overpic}[bb=0 0 90 90, height=50 pt]{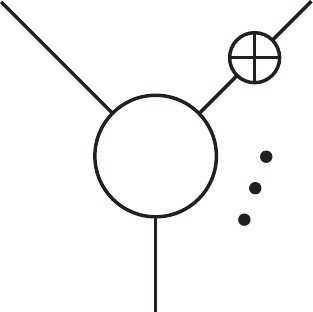}
\put(23,23){$p$}
\end{overpic}} \hspace{0.6cm}
(4) \hspace{0.4cm} \raisebox{-24 pt}{\begin{overpic}[bb=0 0 121 153, height=55 pt]{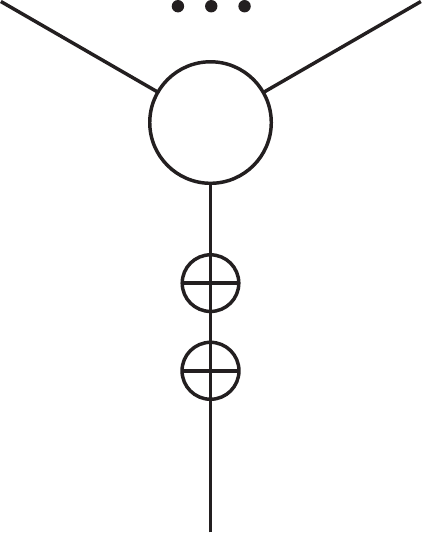}
\put(20,43){$_p$}
\end{overpic}} 
\hspace{0.4cm}\mbox{\LARGE$\leftrightarrow$}\hspace{0.4cm}
\raisebox{-24 pt}{\begin{overpic}[bb=0 0 121 153, height=55 pt]{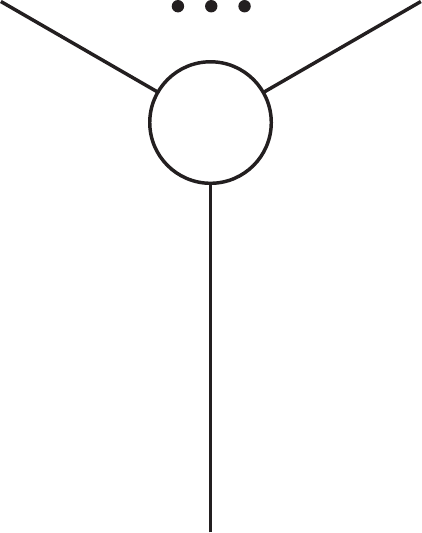}
\put(20,43){$_p$}
\end{overpic}} $$
\vspace{0.3cm}
$$
(5)  \hspace{0.3cm} \raisebox{-30 pt}{\begin{overpic}[bb=0 0 215 232, height=80 pt]{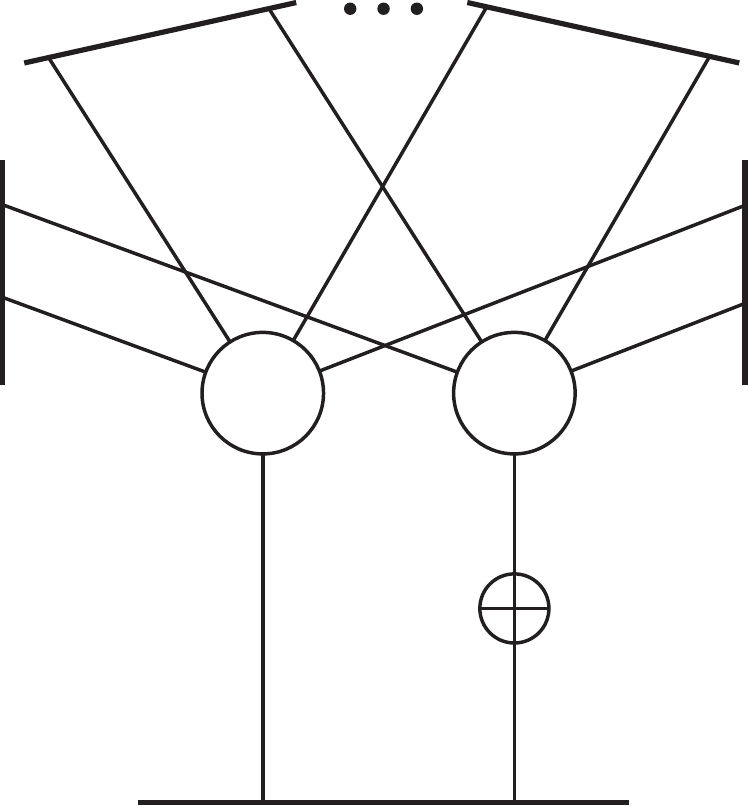}
\put(24,42){$_p$} \put(49,42){$_p$}
\end{overpic}} 
\hspace{0.5cm}\mbox{\LARGE$\leftrightarrow$}\hspace{0.5cm}
\raisebox{-30 pt}{\begin{overpic}[bb=0 0 215 232, height=80 pt]{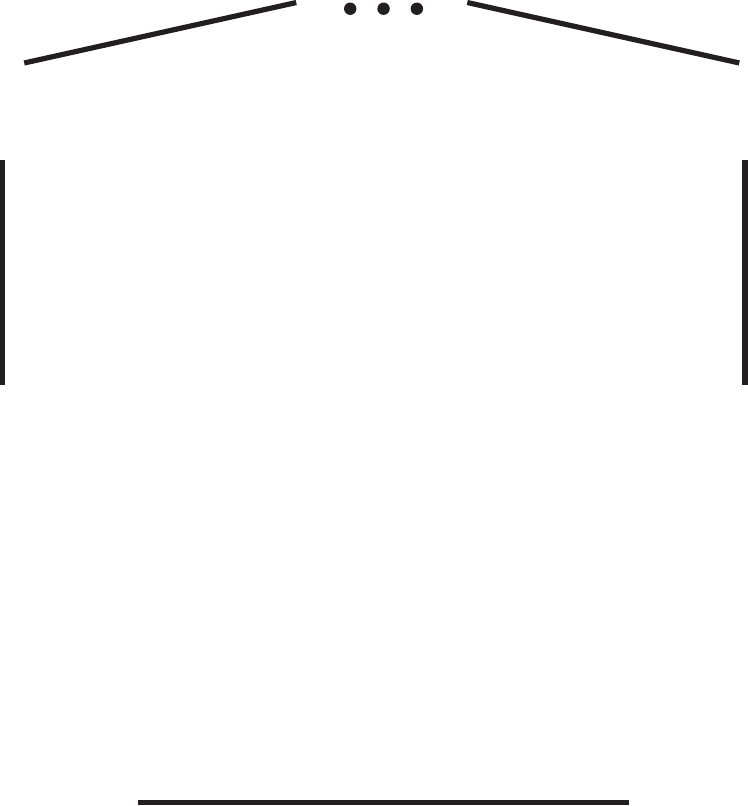}
\end{overpic}}
\hspace{0.3cm}
(6) \hspace{0.4cm} \raisebox{-20 pt}{\begin{overpic}[bb=0 0 159 158, height=65 pt]{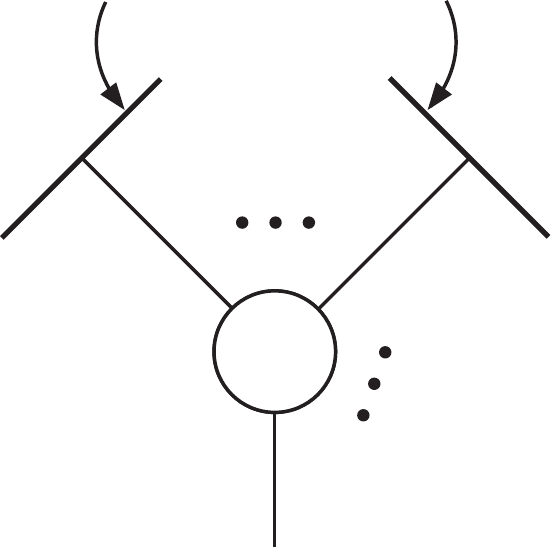}
\put(-7,70){\small same component} \put(31,24){$_p$}
\end{overpic}} 
\hspace{0.5cm}\mbox{\LARGE$\leftrightarrow$}\hspace{0.5cm}
\raisebox{17 pt}{\begin{overpic}[bb=0 0 159 69, height=28 pt]{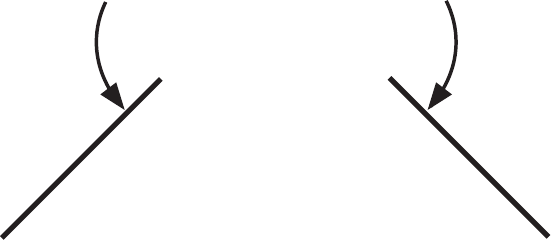}
\put(-7,33){\small same component}
\end{overpic}} 
\hspace{0.6cm}
$$
\vspace{0.5cm}
$$
(7) \hspace{0.4cm} \raisebox{-32 pt}{\begin{overpic}[bb=0 0 215 236, height=85 pt]{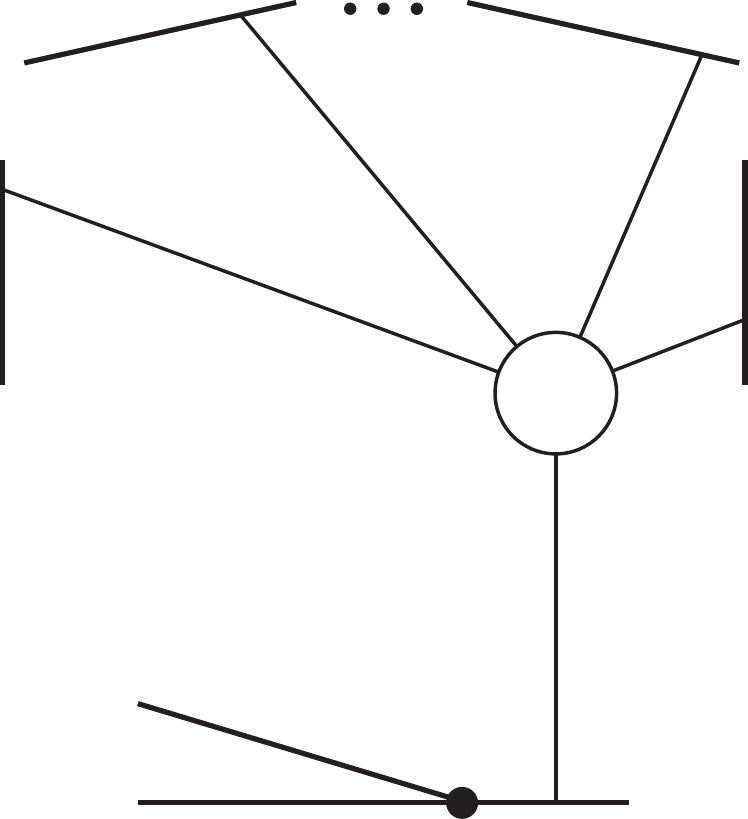}
\put(56,45){$_p$}
\end{overpic}} 
\hspace{0.5cm}\mbox{\LARGE$\leftrightarrow$}\hspace{0.5cm}
\raisebox{-32 pt}{\begin{overpic}[bb=0 0 215 236, height=85 pt]{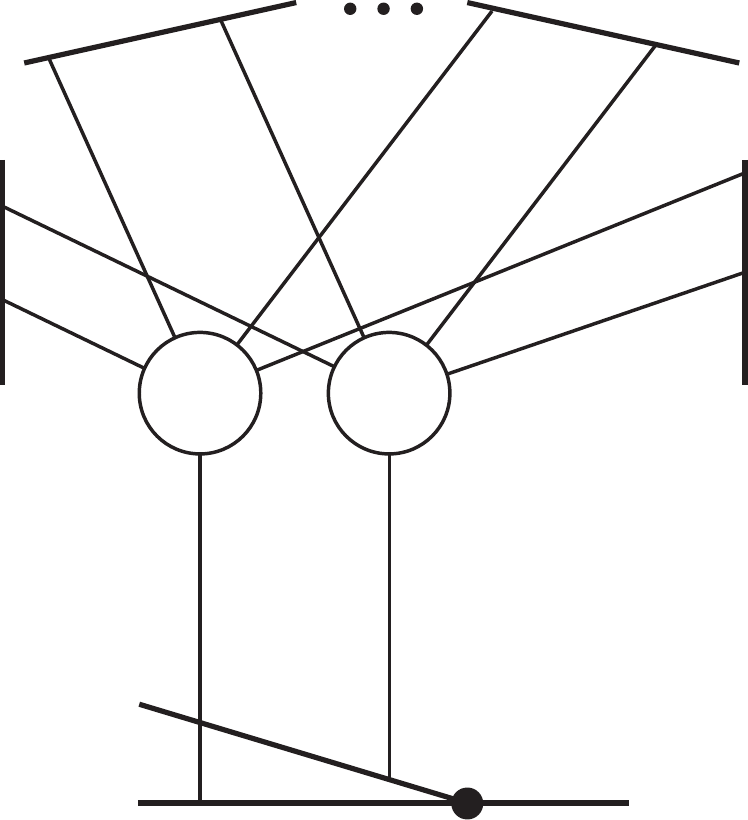}
\put(19,45){$_p$} \put(39,45){$_p$}
\end{overpic}} $$
\caption{Moves of $C^{(p)}_{n-1}$-trees.} \label{MoCnCla}
\end{figure}

\par
Let $H = L_1 \cup \dots \cup L_n$ be an $n$-component almost trivial handlebody-link. Fix a basis $\{e^i_1, \dots, e^i_{g_i} \}$ of $H_1(L_i; \mathbb{Z})$ where $g_i$ is the genus of $L_i$. This determines a basis $\mathcal{B}=\{e^1_1, \dots, e^1_{g_1}, \dots, e^n_1,\dots, e^n_{g_n}\}$ of $H_1(H; \mathbb{Z})$. 
Let $\Gamma_0 = \gamma_1 \cup \dots \cup \gamma_n$ be an oriented plane graph whose components $\gamma_i$ are bouquet graphs with $g_i$ loop edges. Let $\Gamma$ be an oriented bouquet graph presentation of $H$ whose loop edges represent $\mathcal{B}$.
From Lemma \ref{HBLtri}, $\Gamma_0$ is equivalent to $\Gamma$ up to $C_{n-1}$-equivalence and HL-homotopy, and $\Gamma$ is HL-homotopic to a graph $\Gamma'$ obtained by attaching some simple $C_{n-1}$-trees to $\Gamma_0$. 
We assume that the $C_{n-1}$-trees of $\Gamma'$ are all $C_{n-1}^{(p)}$-trees. 
Using the moves (1)-(6) in Figure \ref{MoCnCla} we modify $\Gamma'$ to a canonical form $\Gamma_{\mathcal{B}}$ satisfying that
\begin{enumerate}
\item[(a)] Any $C^{(p)}_{n-1}$-tree connects edges of $n$ different components.
\item[(b)] $C^{(p)}_{n-1}$-trees which connect the same $n$ edges are arranged parallel as in Figure \ref{ParCnCla}, where the bottom line is the $n$-th component, and they are all non-twisted or half-twisted. 
\end{enumerate} 
We remark that by (1) in Figure \ref{MoCnCla}, we do not have to consider over or under information of crossing between two $C_{n-1}^{(p)}$-trees. 
Let $f^i_j$ be an edge of $\gamma_i$ which corresponds to $e^i_j$ in $\Gamma_{\mathcal{B}}$. 
By Lemma~\ref{MIlnorproperty} and \ref{additivity}, the number $r$ of $C^{(p)}_{n-1}$-trees connecting $f^1_{k_1}, \dots, f^n_{k_n}$ is equal to $\overline{\mu}_{e^1_{k_1}\cup \dots \cup e^n_{k_n}}(I_p)$.
\begin{proof}[Proof of Theorem \ref{thm2}] 
\par
The well-definedness of $\varphi$ comes from that of $t_{H}$. 
We construct the inverse map of $\varphi$. Let $t=(t_1, \dots, t_{(n-2)!})$ be an $(n-2)!$-tuple of tensors $t_p$ in  ${\mathbb{Z}}^{m_{1}} \otimes \dots \otimes {\mathbb{Z}}^{m_{n}}$, where $t_p= \sum^{m_1,\dots, m_n}_{k_1, \dots, k_n=1} a^p_{k_1, \dots, k_n} \overline{\boldsymbol{e}}_{{k_1}}^{1} \otimes \dots \otimes \overline{\boldsymbol{e}}_{{k_n}}^{n}$. Let $\Gamma_0$ be an $n$-component oriented plane graph whose components are bouquet graphs $\gamma_i$, where the number of $\gamma_i$'s edges are $m_i$. Let $f^i_j$ be the $j$-th edge of $\gamma_i$. Let $\Gamma_T$ be $\Gamma_0$ with $a^p_{k_1\dots k_n}$ parallel $C^{(p)}_{n-1}$-trees attached to $n$ edges $f^1_{k_1}, \dots, f^n_{k_n}$ for every $p$. 
Then we define a map 
$$\psi: \overline{X}_n \rightarrow  \overline{W}_n; [t] \mapsto [\Gamma_t],$$
where $[t]$ is the element of $\overline{X}_n$ represented by $t$ and $[\Gamma_t]$ is the HL-homotopy class represented by $\Gamma_t$.  
We show the well-definedness of $\psi$. 
The action $\rho$ is generated by the three kinds of elements in $GL(m_{1}, \mathbb{Z}) \times  \dots \times GL(m_{n}, \mathbb{Z})$:
\begin{itemize}
\item[(i)] $(E, \dots, \overset{i}{\check{P(l,h)}} ,\dots, E)$, where $P(l,h)$ is a matrix obtained from the identity matrix by the swapping row $l$ and row $h$.
\item[(ii)] $(E, \dots, \overset{i}{\check{Q(l)}} ,\dots, E)$, where $Q(l)$ is a matrix obtained from the identity matrix by changing 1 in the $l$-th position to $-1$.
\item[(iii)] $(E, \dots, \overset{i}{\check{R(l, h)}} ,\dots, E)$, where $R(l,h)$ is a matrix obtained from identity matrix by adding 1 in the $(l,h)$ position.
\end{itemize}
We focus on the case $i=n$. 
By the diagonal actions of (i), (ii) and (iii) for $t$, the image of $t$ changes as in Figure \ref{TorCla}, where we focus on only $C^{(p)}_{n-1}$-trees and omit the number $p$. 
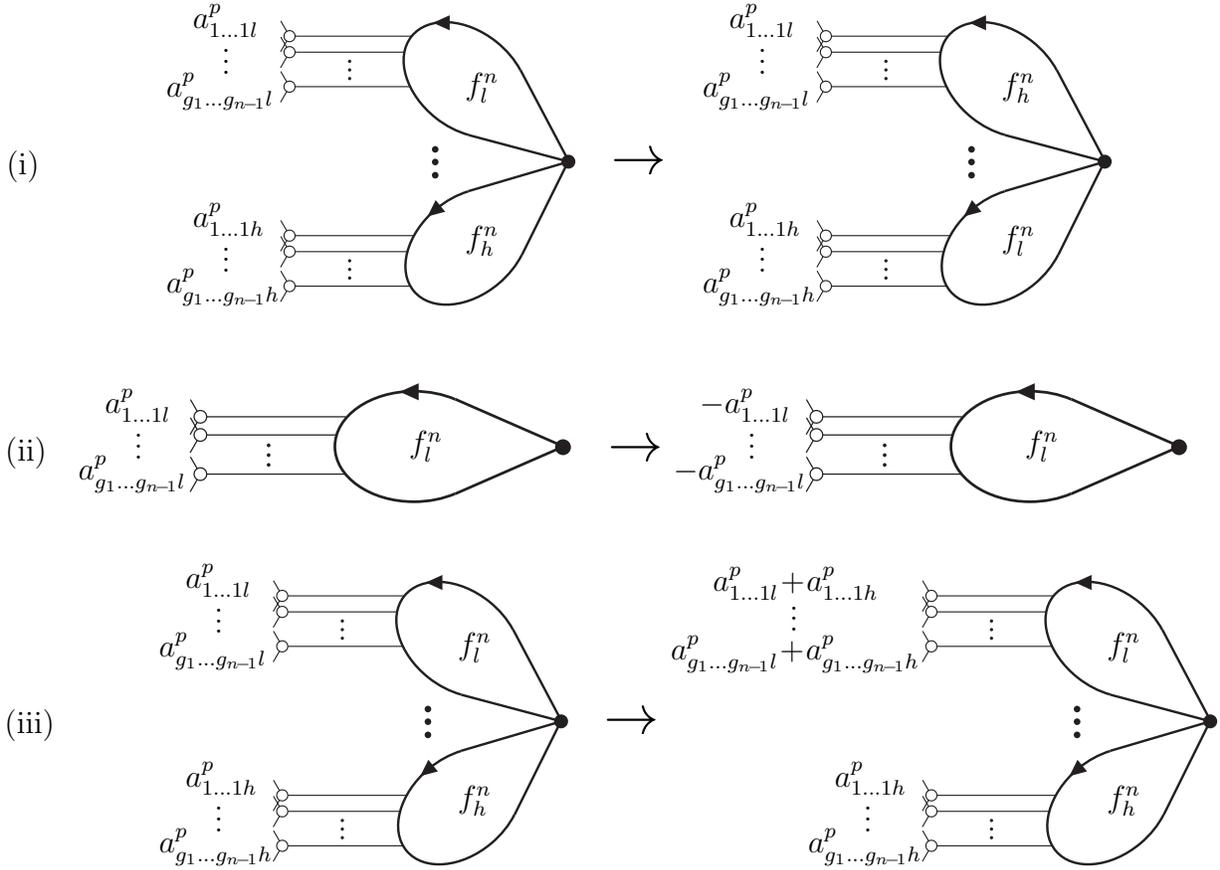
\begin{figure}[ht]
$$  \hspace{0.2cm}\hspace{-1.2cm}(\mbox{i})\hspace{1.2cm} \hspace{0.5cm}  \hspace{1.5cm}\raisebox{-50 pt}{\begin{overpic}[bb=0 0 239 236, height=110 pt]{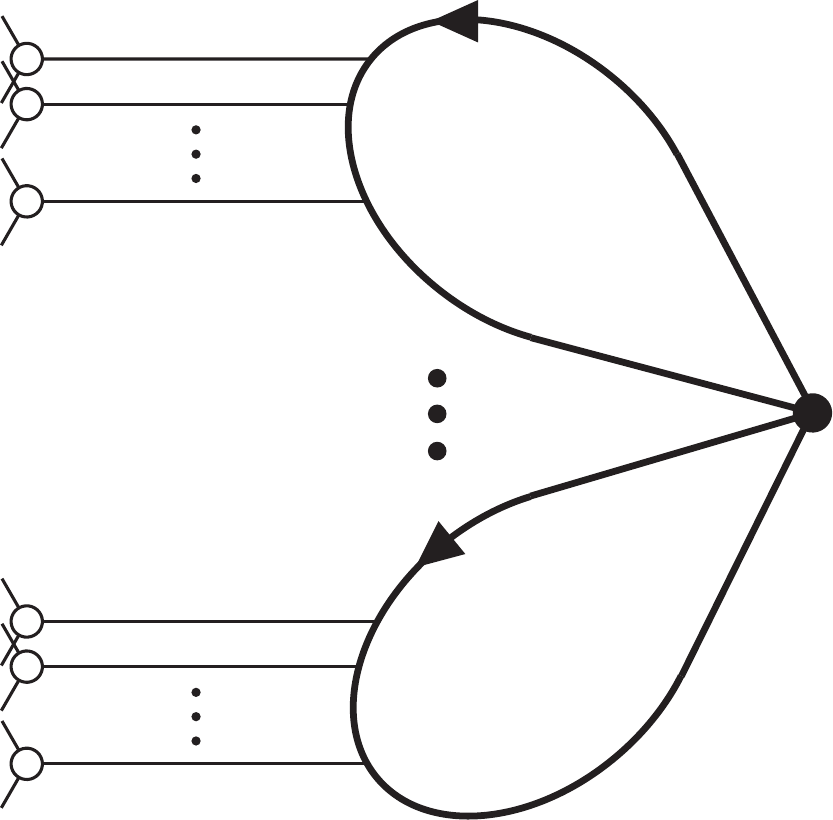}
\put(-33,105){$a^p_{1\dots1l}$} \put(-22,88){$\vdots$} \put(-43,79){$a^p_{g_1 \dots g_{n\!-\!1} l}$}
\put(-33,30){$a^p_{1\dots1h}$} \put(-22,13){$\vdots$} \put(-43,5){$a^p_{g_1 \dots g_{n\!-\!1} h}$}
\put(70, 80){$f^n_l$} \put(70, 22){$f^n_h$}
\end{overpic}} 
\hspace{0.5cm}\mbox{\LARGE$\rightarrow$}\hspace{0.5cm}
 \hspace{1.5cm}\raisebox{-50 pt}{\begin{overpic}[bb=0 0 239 236, height=110 pt]{Bouquet02-2.pdf}
\put(-33,105){$a^p_{1\dots1l}$} \put(-22,88){$\vdots$} \put(-43,79){$a^p_{g_1 \dots g_{n\!-\!1} l}$}
\put(-33,30){$a^p_{1\dots1h}$} \put(-22,13){$\vdots$} \put(-43,5){$a^p_{g_1 \dots g_{n\!-\!1} h}$}
\put(70, 22){$f^n_l$} \put(70, 80){$f^n_h$}
\end{overpic}} $$
\vspace{0.8cm}
$$ (\mbox{ii}) \hspace{0.5cm}  \hspace{1.4cm}\raisebox{-16 pt}{\begin{overpic}[bb=0 0 272 85, height=45 pt]{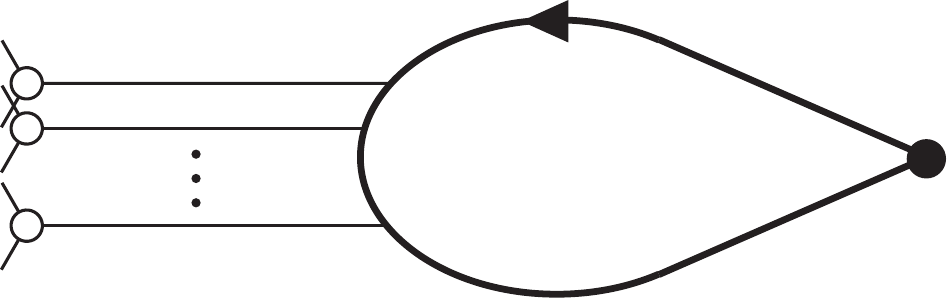}
\put(-32, 34){$a^p_{1\dots1l}$} \put(-21,17){$\vdots$} \put(-42,9){$a^p_{g_1 \dots g_{n\!-\!1} l}$}
\put(83,19){$f^n_l$}
\end{overpic}} 
\hspace{0.5cm}\mbox{\LARGE$\rightarrow$}\hspace{0.5cm}
 \hspace{1.4cm}\raisebox{-16 pt}{\begin{overpic}[bb=0 0 272 85, height=45 pt]{Bouquet01-2.pdf}
\put(-40, 34){$-a^p_{1\dots1l}$} \put(-22,17){$\vdots$} \put(-50,9){$-a^p_{g_1 \dots g_{n\!-\!1} l}$}
\put(83,19){$f^n_l$}
\end{overpic}}$$ 
\vspace{0.8cm}
$$
 \hspace{0.8cm}\hspace{-0.7cm}(\mbox{iii}) \hspace{0.7cm} \hspace{0.5cm}  \hspace{1.7cm}\raisebox{-50 pt}{\begin{overpic}[bb=0 0 239 236, height=110 pt]{Bouquet02-2.pdf}
\put(-33,105){$a^p_{1\dots1l}$} \put(-22,88){$\vdots$} \put(-43,79){$a^p_{g_1 \dots g_{n\!-\!1} l}$}
\put(-33,30){$a^p_{1\dots1h}$} \put(-22,13){$\vdots$} \put(-43,5){$a^p_{g_1 \dots g_{n\!-\!1} h}$}
\put(70, 80){$f^n_l$} \put(70, 22){$f^n_h$}
\end{overpic}} 
\hspace{0.5cm}\mbox{\LARGE$\rightarrow$}\hspace{0.5cm}
 \hspace{3.0cm}\raisebox{-50 pt}{\begin{overpic}[bb=0 0 239 236, height=110 pt]{Bouquet02-2.pdf}
\put(-79,104){$a^p_{1\dots1l}\!+\!a^p_{1\dots1h}$} \put(-50,89){$\vdots$} \put(-95,78){$a^p_{g_1 \dots g_{n\!-\!1} l}\!+\!a^p_{g_1 \dots g_{n\!-\!1} h}$}
\put(-33,30){$a^p_{1\dots1h}$} \put(-22,13){$\vdots$} \put(-43,5){$a^p_{g_1 \dots g_{n\!-\!1} h}$}
\put(70, 80){$f^n_l$} \put(70, 22){$f^n_h$}
\end{overpic}}
$$
\caption{Change of $\Gamma_M$ by the elementary transformations on $M$.} \label{TorCla}
\end{figure}
The transformation (i) is realized by exchanging the order of two elements of the basis. The other two transformations are realized by ambient isotopies, reversing edge-orientations, edge-slides and moves in Lemma \ref{MoCnCla} as in Figure \ref{ReTraCla}. Especially, the move (7) is used in the first move of (iii). Therefore $\psi$ is well-defined. It is obvious that $\psi$ is the inverse map of $\varphi$. 
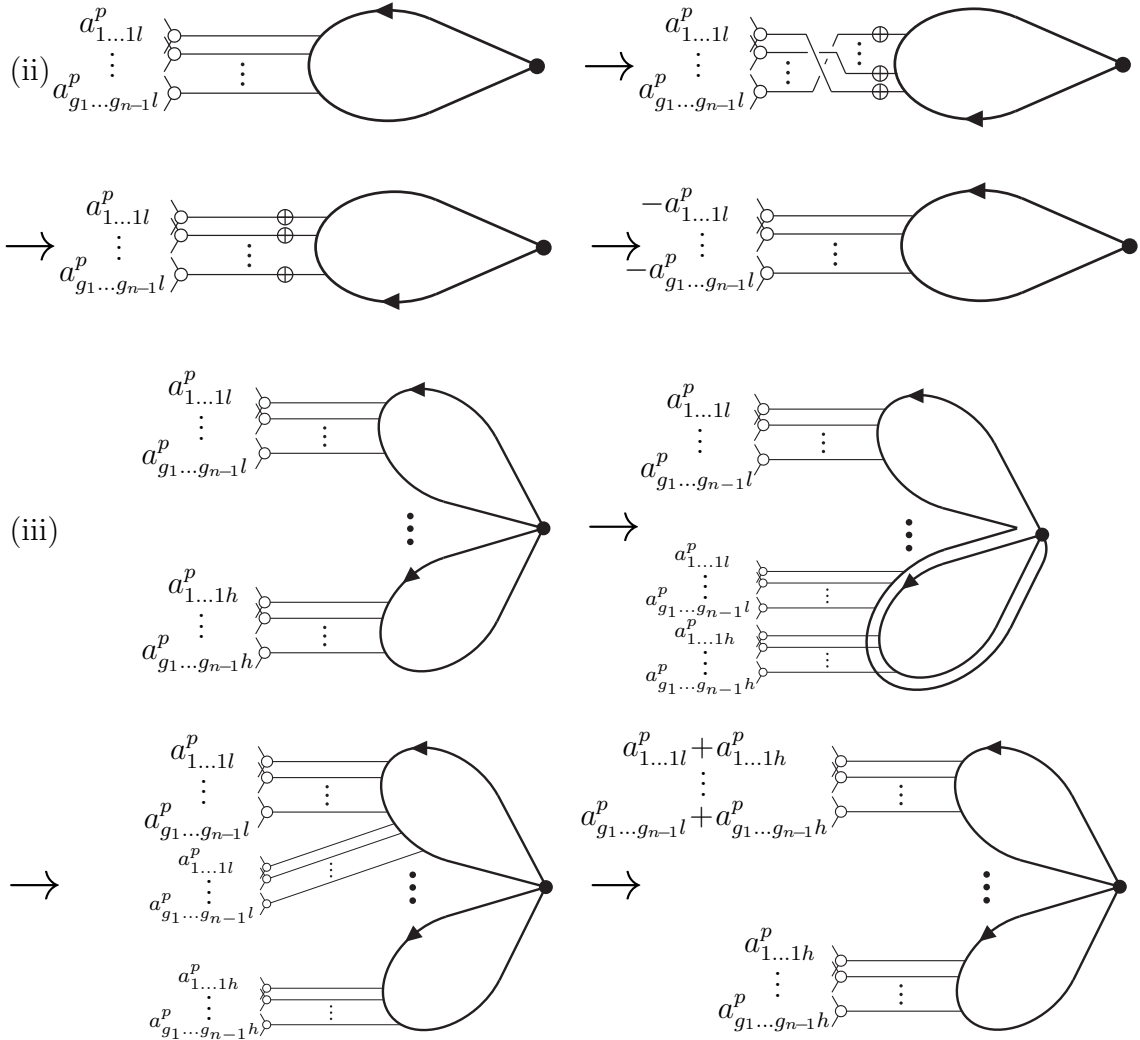
\begin{figure}[ht]
$$ (\mbox{ii}) \hspace{0.5cm}  \hspace{1.0cm}\raisebox{-16 pt}{\begin{overpic}[bb=0 0 272 85, height=45 pt]{Bouquet01-2.pdf}
\put(-32, 34){$a^p_{1\dots1l}$} \put(-21,17){$\vdots$} \put(-42,9){$a^p_{g_1 \dots g_{n\!-\!1} l}$}
\end{overpic}} 
\hspace{0.5cm}\mbox{\LARGE$\rightarrow$}\hspace{0.5cm}
 \hspace{1.0cm}\raisebox{-18 pt}{\begin{overpic}[bb=0 0 272 85, height=45 pt]{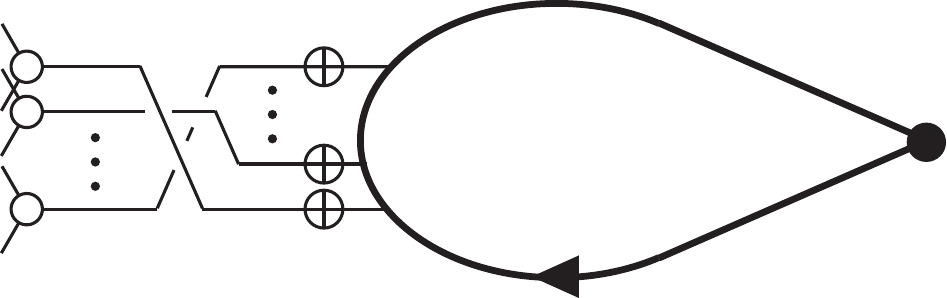}
\put(-32, 36){$a^p_{1\dots1l}$} \put(-21,19){$\vdots$} \put(-42,11){$a^p_{g_1 \dots g_{n\!-\!1} l}$}
\end{overpic}}$$ 
 \vspace{0.5cm}
$$ \mbox{\LARGE$\rightarrow$}\hspace{0.5cm}
\hspace{1.0cm}\raisebox{-19 pt}{\begin{overpic}[bb=0 0 272 85, height=45 pt]{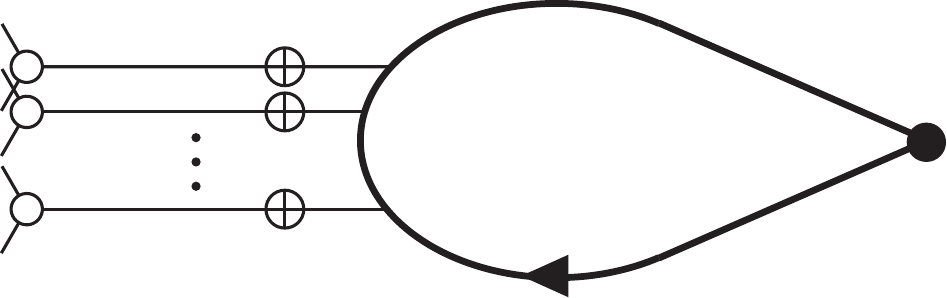}
\put(-32, 36){$a^p_{1\dots1l}$} \put(-21,19){$\vdots$} \put(-42,11){$a^p_{g_1 \dots g_{n\!-\!1} l}$}
\end{overpic}} 
\hspace{0.5cm}\mbox{\LARGE$\rightarrow$}\hspace{0.5cm}
 \hspace{1.0cm}\raisebox{-16 pt}{\begin{overpic}[bb=0 0 272 85, height=45 pt]{Bouquet01-2.pdf}
\put(-44, 34){$-a^p_{1\dots1l}$} \put(-22,17){$\vdots$} \put(-50,9){$-a^p_{g_1 \dots g_{n\!-\!1} l}$}
\end{overpic}}$$ 
 \vspace{0.8cm}
$$\hspace{-1.1cm}(\mbox{iii}) \hspace{1.1cm}  \hspace{0.5cm}  \hspace{1.0cm}\raisebox{-50 pt}{\begin{overpic}[bb=0 0 239 236, height=110 pt]{Bouquet02-2.pdf}
\put(-33,105){$a^p_{1\dots1l}$} \put(-22,88){$\vdots$} \put(-43,79){$a^p_{g_1 \dots g_{n\!-\!1} l}$}
\put(-33,30){$a^p_{1\dots1h}$} \put(-22,13){$\vdots$} \put(-43,5){$a^p_{g_1 \dots g_{n\!-\!1} h}$}
\end{overpic}} 
\hspace{0.5cm}\mbox{\LARGE$\rightarrow$}\hspace{0.5cm}
 \hspace{1.0cm}\raisebox{-57 pt}{\begin{overpic}[bb=0 0 239 236, height=110 pt]{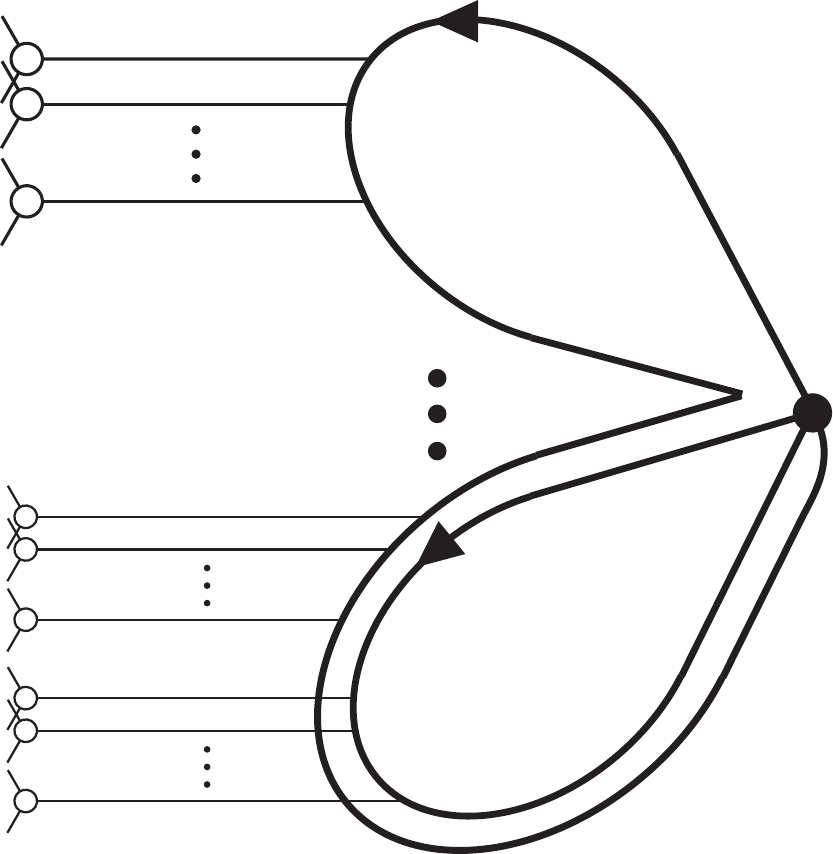}
 \put(-33,107){$a^p_{1\dots 1l}$} \put(-22,90){$\vdots$} \put(-43,81){$a^p_{g_1 \dots g_{n-1} l}$}
\put(-32,52){$\,_{a^p_{1\dots 1l}}$} \put(-22,46){$\,_{\vdots}$} \put(-43,35){$\,_{a^p_{g_1 \dots g_{n-1} l}}$}
\put(-32,23){$\,_{a^p_{1\dots 1h}}$} \put(-22,17){$\,_{\vdots}$} \put(-43,6){$\,_{a^p_{g_1 \dots g_{n-1} h}}$}
\end{overpic}}
$$
 \vspace{0.5cm}
$$\hspace{-0.1cm}\mbox{\LARGE$\rightarrow$} \hspace{0.1cm}\hspace{0.5cm}
\hspace{2.0cm}\raisebox{-52 pt}{\begin{overpic}[bb=0 0 239 236, height=110 pt]{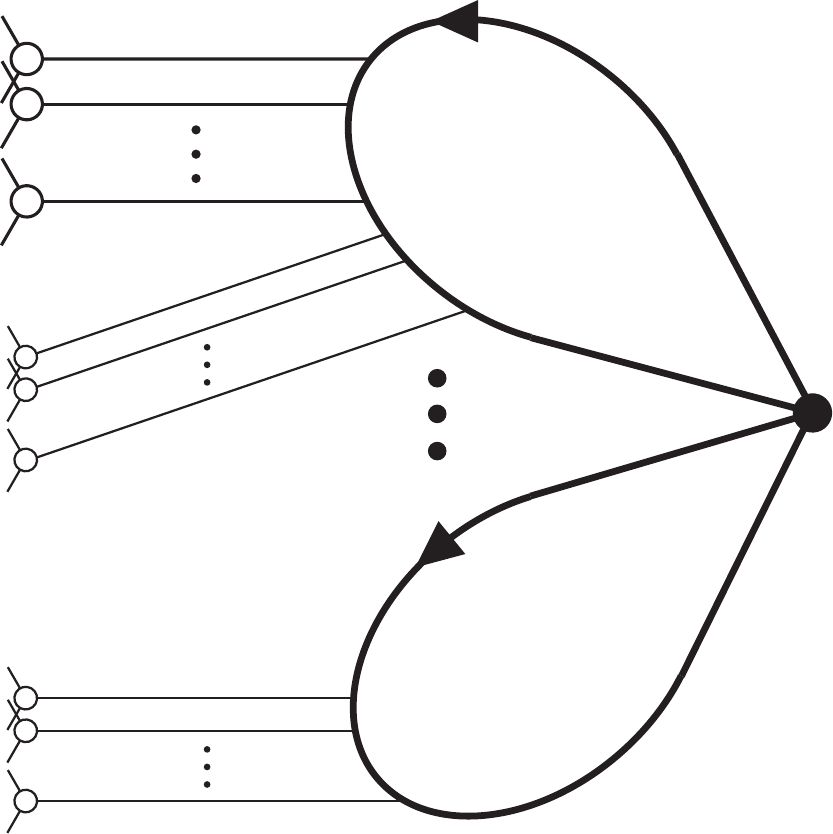}
\put(-33,105){$a^p_{1\dots1l}$} \put(-22,88){$\vdots$} \put(-43,79){$a^p_{g_1 \dots g_{n\!-\!1} l}$}
\put(-32,67){$\,_{a^p_{1\dots 1l}}$} \put(-22,61){$\,_{\vdots}$} \put(-43,50){$\,_{a^p_{g_1 \dots g_{n-1} l}}$}
\put(-32,23){$\,_{a^p_{1\dots 1h}}$} \put(-22,17){$\,_{\vdots}$} \put(-43,6){$\,_{a^p_{g_1 \dots g_{n-1} h}}$}
\end{overpic}} 
\hspace{0.5cm}\mbox{\LARGE$\rightarrow$}\hspace{0.5cm}
 \hspace{2.0cm}\raisebox{-50 pt}{\begin{overpic}[bb=0 0 239 236, height=110 pt]{Bouquet02-2.pdf}
\put(-79,104){$a^p_{1\dots1l}\!+\!a^p_{1\dots1h}$} \put(-50,89){$\vdots$} \put(-95,78){$a^p_{g_1 \dots g_{n\!-\!1} l}\!+\!a^p_{g_1 \dots g_{n\!-\!1} h}$}
\put(-33,30){$a^p_{1\dots1h}$} \put(-22,13){$\vdots$} \put(-43,5){$a^p_{g_1 \dots g_{n\!-\!1} h}$}
\end{overpic}}
$$
\caption{Realization of moves in Figure \ref{TorCla} up to HL-homotopy.} \label{ReTraCla}
\end{figure}
\end{proof}

\appendix
\section{Proof of Theorem \ref{EST}} \label{app1}

\par
We review the proof of Theorem \ref{EST} (\cite{MS}). We prepare some notions.
\begin{definition}[Meridian-disk system]
Let $K$ be a handlebody-knot (i.e. a 1-component handlebody-link). Let $\mathcal{D}$  be a set of mutually disjoint properly embedded disks in $K$ (i.e. mutually disjoint embedded disks whose boundaries are embedded to the boundaries of $K$). We put $B(\mathcal{D}) = cl(K\setminus N(\mathcal{D}; K))$, where $cl(X)$ is the closure of $X$ and $N(X; Y)$ is the regular neighborhood of $X$ in $Y$.  $\mathcal{D}$ is called a \textit{meridian-disk system} of $K$ if $B(\mathcal{D})$ is a set of 3-balls. 
 A meridian-disk system $\mathcal{D}$ is called \textit{complete} if $B(\mathcal{D})$ is one ball.
\end{definition}
For a handlebody-link $H$, by taking the dual of a meridian-disk systems of each component of $H$, 
we have a spatial graph presentation of $H$. Especially, 
by taking the dual of a complete meridian-disk system of each component of $H$, we have a bouquet graph presentation of $H$ (Figure \ref{fig1-13}). 
\begin{figure}
$$\raisebox{0 pt}{\includegraphics[bb=0 0 285 270, height=80 pt]{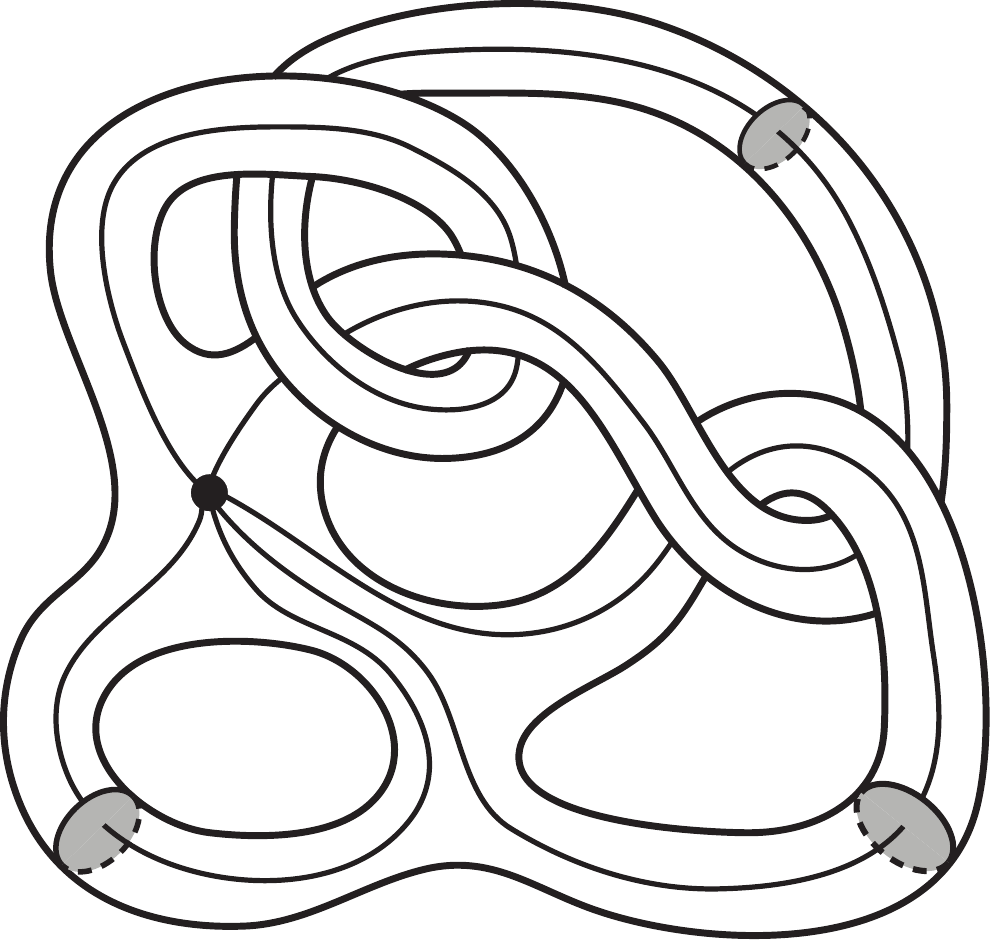}}$$
\vspace{-0.7cm}
\caption{A complete meridian system and its dual graph.} \label{fig1-13}
\end{figure}

\begin{definition}[Disk-sliding move \cite{Su3, Joh}]
Let $\displaystyle \mathcal{D} = \bigcup_{j=1}^{n} D_j$ be a complete meridian-disk system of a genus $n$ handlebody-knot $K$ and let $\alpha$ be a simple arc on $\partial K$ such that $\alpha \cap \mathcal{D} = \partial \alpha$ and $\alpha$ connects two different disks $D_i$ and $D_j$ $\in \mathcal{D}$ (Figure \ref{fig1-15} left). Then, 
$$cl(\partial N(D_i \cup \alpha \cup D_j; K) \setminus \partial K)$$
consists of three proper disks $D'_i$, $D'_j$ and $D_{\alpha}$ (Figure \ref{fig1-15} middle), where $D'_i$ and $D'_j$ are parallel to $D_i$ and $D_j$ respectively. The two meridian-disk systems $$\mathcal{D'} = (\mathcal{D} \setminus D_i) \cup D_{\alpha} \mbox{ and } \mathcal{D''} = (\mathcal{D} \setminus D_j) \cup D_{\alpha}$$ 
are again complete. A \textit{disk-sliding move} of $D_j$ (resp. $D_i$) over $\alpha$ across $D_i$ (resp. $D_j$) is replacing $D_j$ (resp. $D_i$) by $D_{\alpha}$ (Figure \ref{fig1-15} right). 
\begin{figure}[ht]
$$\raisebox{-20 pt}{\begin{overpic}[bb=0 0 266 153, height=48 pt]{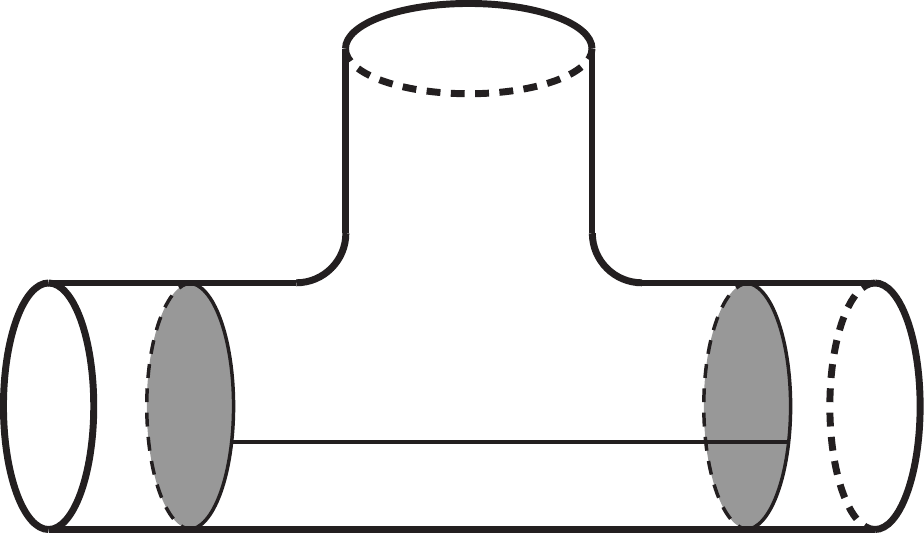} 
\put(11, 28){$D_i$}
\put(63, 28){$D_j$}
\put(41, 12){$\alpha$}
\end{overpic}}
\hspace{0.5 cm}\mbox{\LARGE{$\rightarrow$}}\hspace{0.5 cm}
\raisebox{-20 pt}{\begin{overpic}[bb=0 0 266 153, height=48 pt]{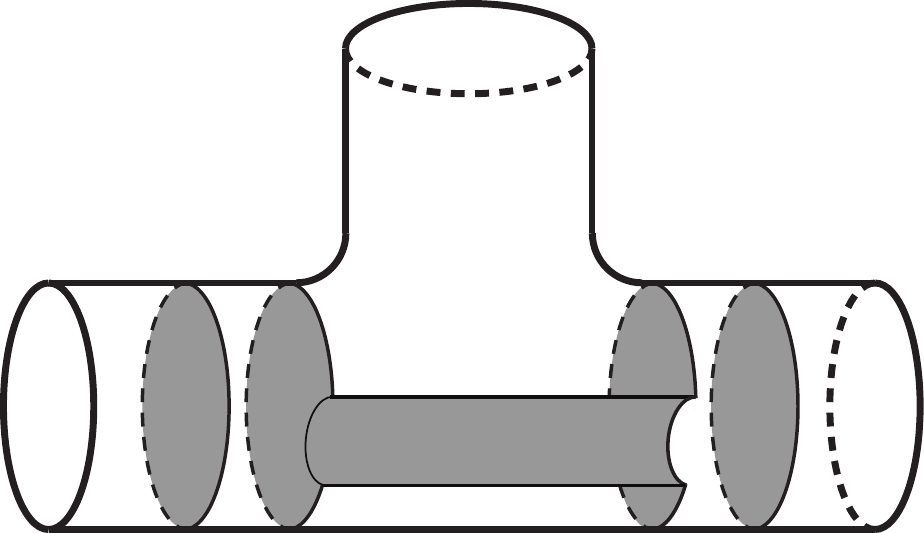}
\put(11, 28){$D'_i$}
\put(63, 28){$D'_j$}
\put(37, 16){$D_{\alpha}$}
\end{overpic}} 
\hspace{0.5 cm}\mbox{\LARGE{$\rightarrow$}}\hspace{0.5 cm}
\raisebox{-20 pt}{\begin{overpic}[bb=0 0 266 153, height=48 pt]{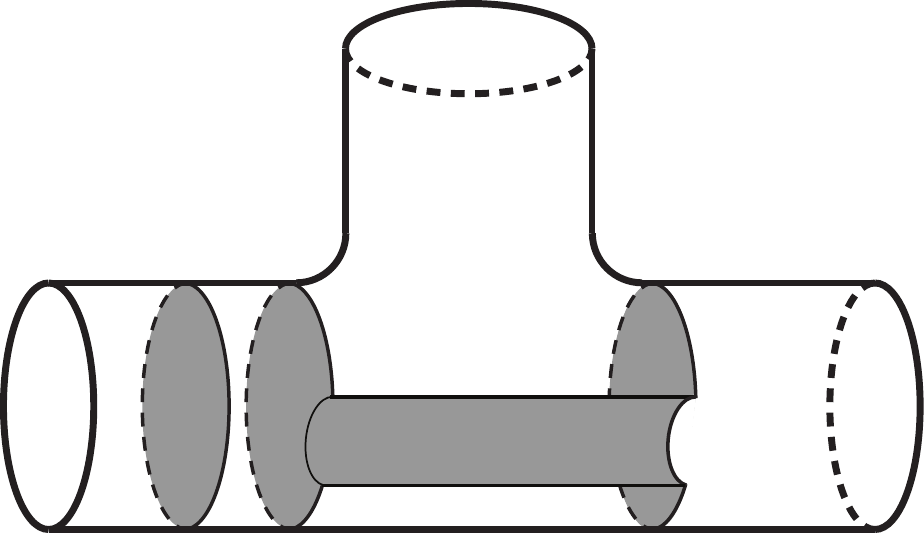}
\put(11, 27){$D_i$}
\put(37, 16){$D_{\alpha}$}
\end{overpic}}$$
\caption{A disk-sliding move.} \label{fig1-15}
\end{figure}
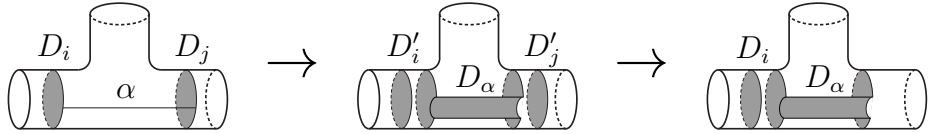
\end{definition}
\begin{lemma}[\cite{Su3, Joh}] \label{lem1-02}
Let $\mathcal{D}$ and $\mathcal{D'}$ be two complete meridian-disk systems of a handlebody-knot $K$, then $\mathcal{D'}$ is obtained from $\mathcal{D}$ by a finite sequence of disk-sliding moves.
\end{lemma}

\begin{proof}[Proof of Theorem \ref{EST}] We prove the handlebody-knot case. This case immediately induces the handlebody-link case. 
\par
For two bouquet graph presentations $\Gamma$ and $\Gamma'$ of a handlebody-knot $K$, the dual meridian-disk systems $\mathcal{D}$ and $\mathcal{D}'$ are both complete. From Lemma \ref{lem1-02}, there is a finite sequence of disk-sliding moves between the two meridian-disk systems. The dual of a disk-sliding move is an edge-sliding move of a bouquet graph (Figure \ref{fig1-17}).
\begin{figure}[ht]
$$\raisebox{-27 pt}{\includegraphics[bb=0 0 336 215, height=60 pt]{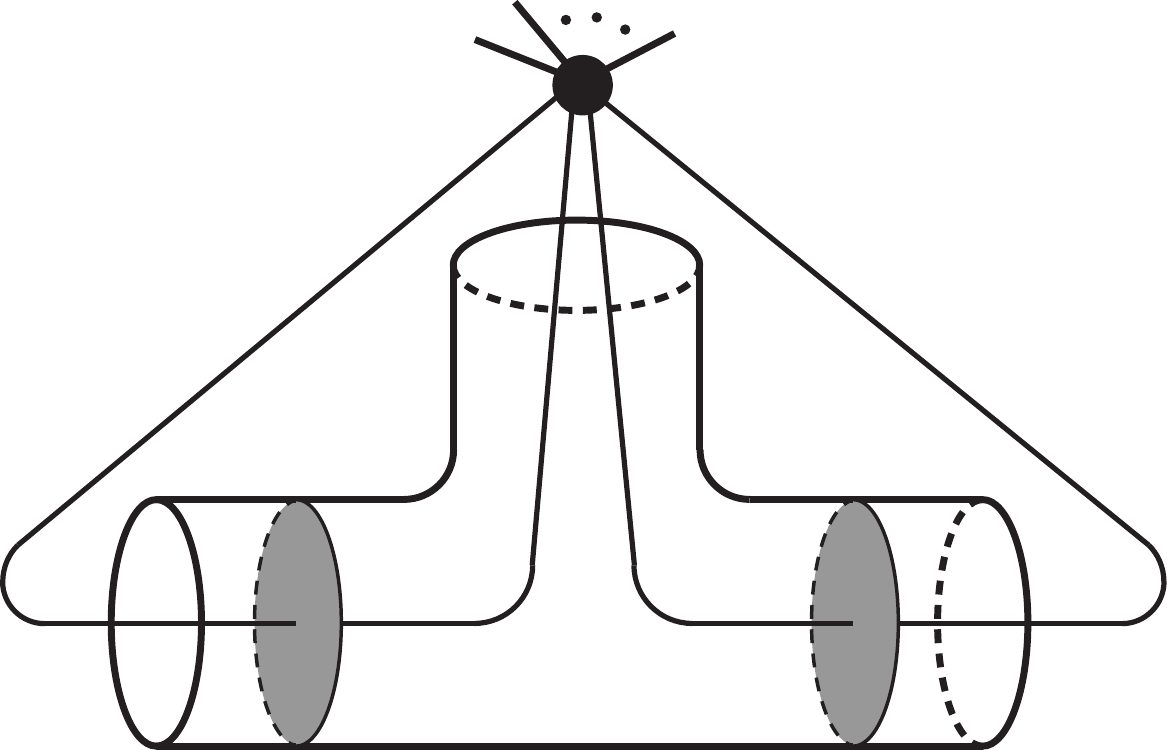}}
\hspace{0.3 cm}\mbox{\LARGE{$\rightarrow$}}\hspace{0.3 cm}
\raisebox{-27 pt}{\includegraphics[bb=0 0 332 208, height=60 pt]{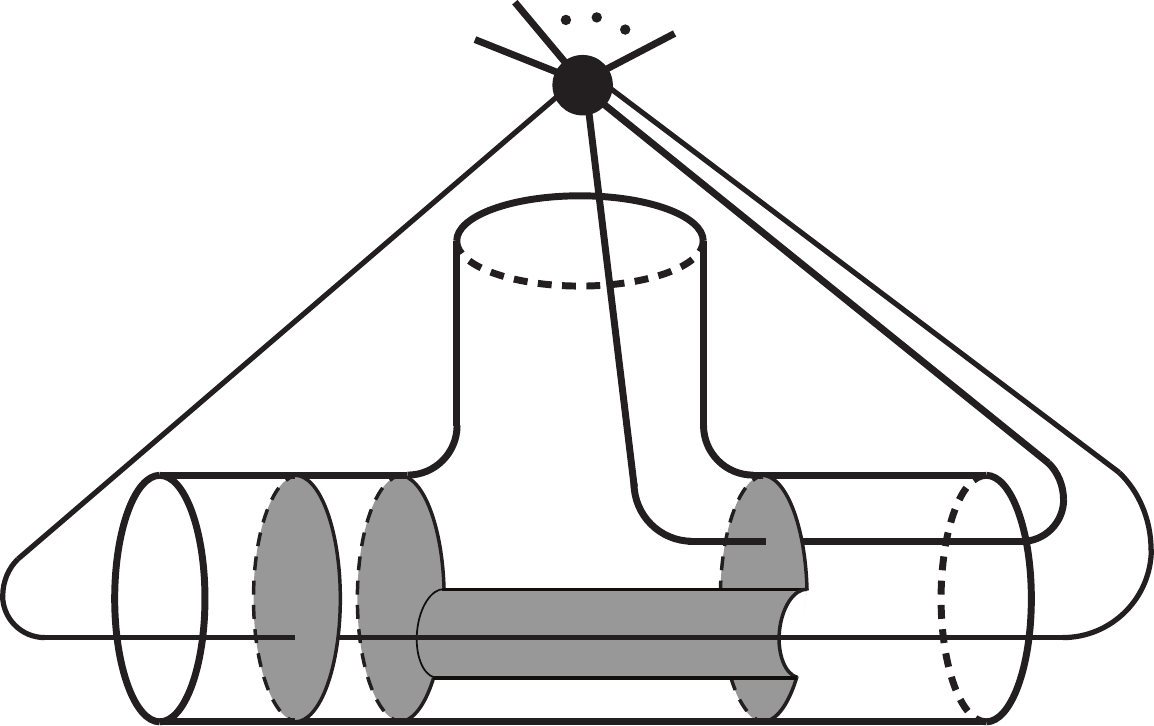}}$$
\caption{A dual of a disk-sliding move.} \label{fig1-17}
\end{figure}
Taking the dual of the sequence of disk-sliding moves, we have a sequence of edge-sliding moves from $\Gamma$ to $\Gamma'$.
\end{proof}

\end{document}